\documentclass{article}
\usepackage[utf8]{inputenc}
\usepackage{cite}
\usepackage{hyperref}

\title{Two-fold quasi-alternating links,\\ Khovanov homology
and instanton homology}
\author{Christopher Scaduto {\;}\& Matthew Stoffregen}
\date{}

\usepackage[a4paper, total={6in, 8.5in}]{geometry}
\usepackage{graphicx}
\usepackage{amssymb}
\usepackage{tikz-cd}
\usepackage{titling}
\usepackage{booktabs}
\usepackage{amsthm}
\usepackage{amscd}
\usepackage{caption}
\usepackage{amsmath}
\usepackage{MnSymbol}
\usepackage{tikz}
\usepackage{dsfont}

\usepackage{sectsty}

\sectionfont{\fontsize{12}{15}\selectfont}

\usetikzlibrary{matrix}

\usetikzlibrary{shapes,snakes}
\usetikzlibrary{matrix,arrows,decorations.pathmorphing}
\usetikzlibrary{positioning}

\usetikzlibrary{matrix,arrows,decorations.pathmorphing}
\usetikzlibrary{positioning}

\newcommand{\R}{\mathbb{R}}

\newcommand{\Z}{\mathbb{Z}}
\newcommand{\Q}{\mathbb{Q}}

\newcommand{\kh}{\text{Kh}}

\newcommand{\F}{\mathbb{F}}

\newtheorem{prop}{Proposition}
\newtheorem{theorem}{Theorem}

\newtheorem{defn}{Definition}
\newtheorem{conjecture}{Conjecture}
\newtheorem{corollary}{Corollary}
\newtheorem{remark}{Remark}
\newtheorem{example}{Example}

\begin{document}

\maketitle 

\begin{abstract}
We introduce a class of links strictly containing quasi-alternating links for which mod 2 reduced Khovanov homology is always thin. We compute the framed instanton homology for double branched covers of such links. Aligning certain dotted markings on a link with bundle data over the branched cover, we also provide many computations of framed instanton homology in the presence of a non-trivial real 3-plane bundle. We discuss evidence for a spectral sequence from the twisted Khovanov homology of a link with mod 2 coefficients to the framed instanton homology of the double branched cover. We also discuss the relevant mod 4 gradings.
\end{abstract}

\section{Introduction}

Quasi-alternating (QA) links are a generalization of alternating links introduced by Ozsv\'{a}th and Szab\'{o} \cite{ozsz}. Manolescu and Ozsv\'{a}th \cite{mo} showed that the reduced Khovanov homology with integer coefficients of a QA link is determined by classical link invariants: it is a free abelian group whose rank is the determinant of the link, and the graded decomposition of this group is determined by the Jones polynomial. Here we introduce {\emph{Two-fold quasi-alternating {\emph{(TQA)}} links}}, a class strictly containing that of QA links. For TQA links, Khovanov homology with mod 2 coefficients is similarly determined by classical invariants.

The authors were led to the class of TQA links through studying instanton homology.
In \cite{scaduto}, the first author constructed a spectral sequence with $E^2$-page the reduced odd Khovanov homology of a link converging to the framed instanton homology of its double branched cover equipped with a trivial $SO(3)$-bundle. In this article our attention turns to non-trivial bundles. We first compute the framed instanton homology for double branched covers of TQA `links with $SO(3)$-bundle,' or {\emph{two-fold marked links}}. We then produce a mod 4 graded spectral sequence from a combinatorially defined homology theory for `dotted diagrams' that converges to the framed instanton homology of the double branched cover, with bundle data determined by the dots.\\

\vspace{.25cm}

\noindent {\textbf{A generalized notion of quasi-alternating.} To state our results, we first introduce the primary objects of interest. Let  $L$ be a link, and let $\omega$ assign $0$ or $1$ to each component of $L$. We think of $\omega$ pictorially as assigning dots to our diagram: one dot on each component that $\omega$ assigns $1$. We require that there are an even number of dots in total. We call $\omega$ {\emph{two-fold marking data for $L$}}. We say that $\omega$ is {\emph{non-trivial}} if it assigns 1 to at least one component of $L$. Note that if $L$ is a knot, it only supports trivial two-fold marking data. We call the pair $(L,\omega)$ a {\emph{two-fold marked link}}. 

The link $L$ has $2^{|L|-1}$ distinct choices for two-fold marking data, where $|L|$ is the number of components of $L$. See Figure \ref{fig:2foldmarkings}. It will be convenient, in the sequel, to allow for more than one dot on each component of $L$, in which case $\omega$ counts dots mod 2. Thus a two-fold marked link is a link with dots, in which one imposes the relation that two dots on the same component cancel. Given a component $K$ of the link $L$ we write $\omega(K)\in\Z_2$ for the value that $\omega$ assigns to $K$.

Let $D$ be a planar diagram representing $L$. By an {\emph{arc}} of $D$ we mean a strand of $D$ that descends to an edge of the 4-valent graph formed from $D$ upon forgetting its crossings. Write $\Gamma$ for the set of all arcs. Let $K$ be a component of $L$, and $\Gamma(K)$ denote the set of arcs in $D$ that form $K$. Choose an assignment $\check{\omega}:\Gamma\to \Z_2$ that is compatible with $\omega$ in the following sense:
\[
	\sum_{\gamma\in \Gamma(K)} \check{\omega}(\gamma) \equiv \omega(K) \mod 2,
\]
for each component $K$.
We call $\check{\omega}$ two-fold marking data for $D$ compatible with $\omega$. This data may be packaged diagrammatically by placing an odd number of dots on any arc $\gamma$ for which $\check{\omega}(\gamma)=1$. We call $(D,\check{\omega})$ a {\emph{two-fold marked diagram}}, and sometimes simply a {\emph{dotted diagram}}, or sometimes even a {\emph{diagram of}} $(L,\omega)$.

Suppose that at a fixed crossing of $D$ we perform a 0-smoothing to obtain a diagram $D_0$. Then we obtain two-fold marking data $\check{\omega}_0$ for $D_0$ by carrying the dots along in the natural way, and by counting dots modulo 2. The two-fold marked diagram $(D_0,\check{\omega}_0)$ then represents a two-fold marked link $(L_0,\omega_0)$. However, it is important to note that the resulting $\omega_0$ depends on $(D,\check{\omega})$, not just $(L,\omega)$ with a local modification of a crossing. See Figure \ref{fig:resolutiondependence}. Similar remarks hold for 1-smoothings.

When $(L,\omega)$ is as above but $\omega$ decorates $L$ with odd number of dots, it is not, by definition, a two-fold marked link. We instead say that $(L,\omega)$ is an {\emph{odd-marked link}}, a notion that will only be auxiliary for us. Recall that a link $L$ is {\emph{split}} if there is a 2-sphere disjoint from $L$ that separates the components of $L$ into two non-empty sets. The following is a generalization of quasi-alternating (QA) links as defined in \cite[Def. 3.1]{ozsz} to our setting of two-fold marked links:\\

\begin{figure}[t]
\centering
\includegraphics[scale=.25]{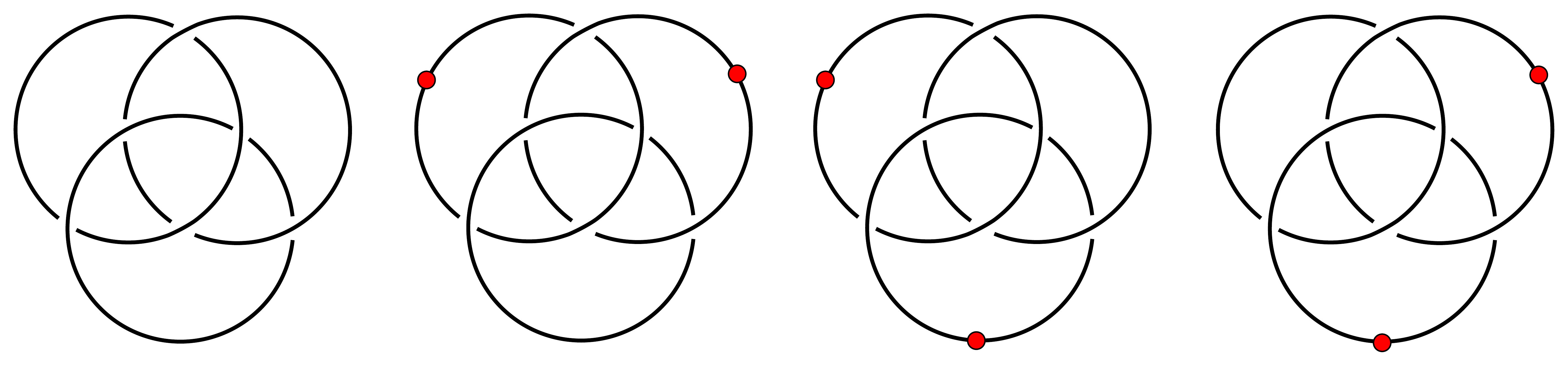}
\caption{The $2^{3-1}=4$ distinct two-fold markings of the 3-component link L6n1.}\label{fig:2foldmarkings}
\end{figure}

\begin{defn}\label{defn:tqa}
	The set of {\emph{two-fold quasi-alternating (TQA) two-fold marked links}}, denoted $\mathcal{Q}_2$, is the smallest set of two-fold marked links satisfying the following:
\begin{itemize}
	\item[(1)] The unknot with its unique trivial two-fold marking data is in $\mathcal{Q}_2$.
	\item[(2)] Any two-fold marked link that splits into two odd-marked links is in $\mathcal{Q}_2$.
	\item[(3)] Suppose $(L,\omega)$ has a diagram $(D,\check{\omega})$ such that the two smoothings $(D_0,\check{\omega}_0)$ and $(D_1,\check{\omega}_1)$ at a crossing represent $(L_0,\omega_0)$ and $(L_1,\omega_1)$, respectively, and that the following hold:
	\begin{itemize}
		\item[$\bullet$] Both smoothings $(L_0,\omega_0)$ and $(L_1,\omega_1)$ are in $\mathcal{Q}_2$.
		\item[$\bullet$] The additivity relation $\det(L)=\det(L_0)+\det(L_1)$ holds.
	\end{itemize}
	Then $(L,\omega)$ is in $\mathcal{Q}_2$.
\end{itemize}
We say that a link $L$ is {\emph{TQA}} if for the trivial two-fold marking data $\omega$ we have $(L,\omega)\in \mathcal{Q}_2$.\\
\end{defn}
\vspace{.25cm}

\noindent We remark that the additivity of determinants automatically holds and need not be checked if one of $(L_0,\omega_0)$ or $(L_1,\omega_1)$ is split into odd-marked links. Note also that if $L$ is QA then $(L,\omega)$ is TQA for any two-fold marking data $\omega$. However, consider $L=$L6n1, the link with the lowest crossing number which is not QA. The three two-fold marked links $(L,\omega)$ with $\omega$ non-trivial, depicted in Figure \ref{fig:2foldmarkings}, are all TQA, as illustrated in Figure \ref{fig:L6n1isinQ2}. On the other hand, L6n1 is not TQA as a link.

\begin{figure}[t]
\centering
\includegraphics[scale=.65]{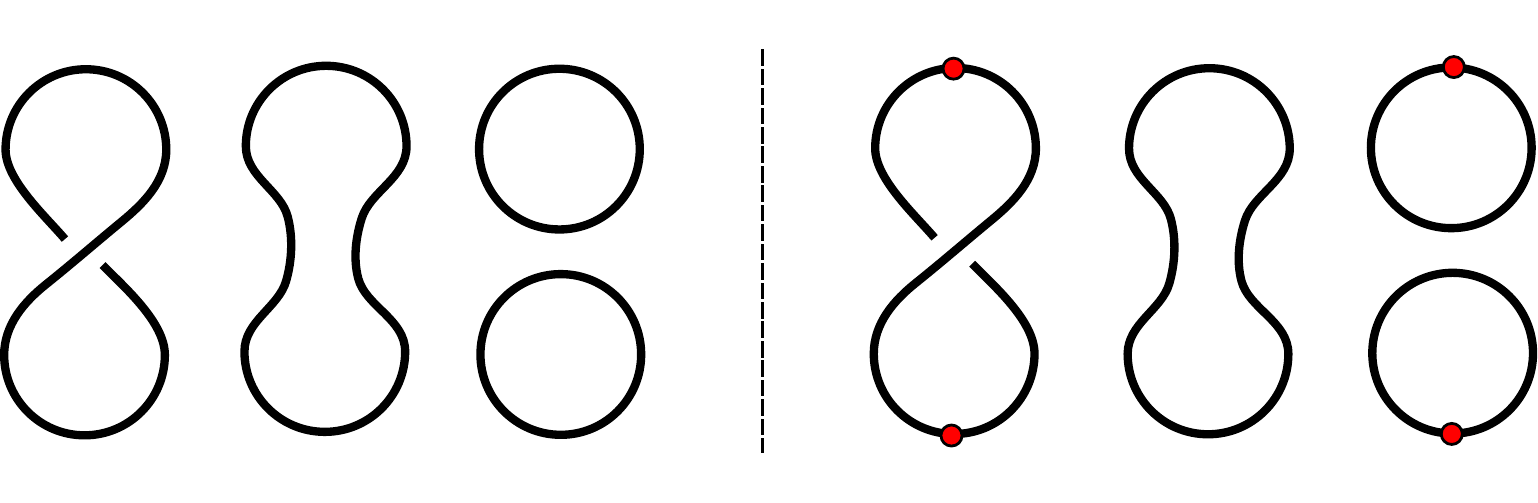}
\caption{In these two figures we perform $0$- and $1$-smoothings on $(L,\omega)$ where $L$ is an unknot and $\omega$ is trivial. On the left, we represent $(L,\omega)$ by the two-fold marked diagram $(D,\check{\omega})$ in which $D$ is an unknot with a twist, and where $\check{\omega}$ is trivial. On the right, we represent $(L,\omega)$ by the two-fold marked diagram $(D,\check{\omega}')$, where $\check{\omega}'$ has a dot on each of the two arcs. The $0$-smoothing in either case is an unknot (which must have trivial marking data), and the $1$-smoothing is on the left an unlink with trivial marking data, and on the right an unlink split into two odd-marked unknots.} \label{fig:resolutiondependence}
\end{figure}

In the next section, we will introduce the $\delta$-graded twisted Khovanov homology $\text{Kh}(L,\omega)$ of a two-fold marked link $(L,\omega)$ with $\Z_2=\Z/2$ coefficients. This is a variant of Roberts' Totally Twisted Khovanov homology \cite{roberts} and the more general construction that appears in Jaeger \cite{jaeger}. It is also related to Baldwin, Levine and Sarkar's Khovanov homology for pointed links \cite{bls}. When $\omega$ is trivial, this invariant is none other than the usual reduced Khovanov homology of $L$, denoted in this article as $\overline{\text{Kh}}(L;\Z_2)$. The following justifies our consideration of TQA links:\\

\begin{theorem}
 \label{prop:tqa}
    If $(L,\omega)$ is TQA, then $\text{{\emph{Kh}}}(L,\omega)$ is supported in $\delta$-grading $0$ and has rank $\det(L)$.\\
\end{theorem}

\noindent In particular, if a link $L$ is TQA, then it is {{\emph{mod 2 Khovanov thin}}. Thus the class of TQA links, which strictly contains that of QA links, helps to further explain the prevalence of mod 2 Khovanov thin links. The proof of this theorem is a straightforward extension of the usual skein exact triangle argument for quasi-alternating links in the framework of our twisted Khovanov homology. The only substantial addition to the proof is an observation that follows readily from the definitions: $\text{Kh}(L,\omega)$ vanishes when $(L,\omega)$ splits into two odd-marked links.

We have the following for low-crossing knots:\\

\begin{itemize}
    \item For prime knots with up to 8 crossings, TQA and QA are equivalent notions.
    \item For prime knots with up to 10 crossings, there are two knots, 9$_{46}$ and 10$_{140}$, which are TQA but not QA. These are the (3,3,-3) and (3,4,-3) pretzel knots, respectively. By the computations of Shumakovitch \cite{shumakovitch}, each has thick odd Khovanov homology, so are not QA.
\end{itemize}

\vspace{.25cm}

\noindent It has been conjectured by Greene \cite{greene} that there exist finitely many QA links of a fixed determinant. This is not true for TQA links. Greene and Watson \cite{greenewatson} constructed an infinite family of thin, hyperbolic, non-QA knots with identical homological invariants. Here, ``thin'' means indistinguishable from a QA knot using the ``homological invariants'' of reduced Khovanov homology, reduced odd Khovanov homology, and Heegaard-Floer knot homology. Their family is a subset of the Kanenobu knots $K_{p,q}$ with $p,q\in \Z$. In Figure \ref{fig:kanenobu} we show that every $K_{p,q}$ is TQA. In Section \ref{sec:more} we will present many TQA non-QA Montesinos knots. We do not have an example of a multi-component link that is TQA but not QA, but have no reason to suspect that they do not exist.\\

\vspace{.55cm}

\noindent{\textbf{Framed instanton homology for TQA branched covers.}} Let $L$ be a link, and let $\Sigma(L)$ be the double cover of $S^3$ branched over $L$. We have the following one-to-one correspondence:

\begin{equation}
 \left\{\begin{array}{c} $SO(3)$-\text{bundles over }\Sigma(L) \\ \text{up to isomorphism} \end{array}\right\}
 \;\; \overset{1:1}\longleftrightarrow \;\;  \left\{\begin{array}{c} \text{Two-fold marking }\\ \text{data for the link } L\end{array}\right\} \label{eq:bundlemarking}
\end{equation}
\vspace{.10cm}

\noindent The correspondence is as follows. Given two-fold marking data $\omega$ for $L$, pair up the dots on $L$ determined by $\omega$ in any fashion. Then consider a collection of embedded arcs in $S^3$ whose interiors are disjoint from $L$, and whose boundaries together form the set of dots determined by $\omega$. Now, lift the arcs to $\Sigma(L)$ to obtain a union of embedded loops. The desired $SO(3)$-bundle $E$ is one for which $w_2(E)$ is Poincar\'{e} dual to this union.

The class of TQA two-fold marked links was discovered by the authors as the natural one for the following result, which gives computations of the $\Z_4$-graded framed instanton homology $I^\#(\Sigma(L),\omega)$, with bundle data determined by $\omega$ as in (\ref{eq:bundlemarking}). (For a quick review of framed instanton homology, see Section \ref{sec:framed}.) Unlike the previous result, the following holds with integer coefficients.\\
\vspace{.15cm}

\begin{theorem} \label{thm:tqa}
	If $(L,\omega)$ is TQA then  $I^\#(\Sigma(L),\omega)$ is free abelian of rank $\det(L)$ and is supported in gradings $\{0,2\}\subset \Z_4$. If $\omega$ is non-trivial, then the ranks in degrees $0$ and $2$ are both
\[
	\frac{1}{2}\det(L).
\]
If $\omega$ is trivial, then the rank in grading $j \in \{0,2\}\subset \Z_4$ is given by
\[
	\frac{1}{2}\left(\det(L) + (-1)^{( j/2 )}2^{|L|-1}\right).
\]
where $|L|$ is the number of components of $L$.\\
\end{theorem}

\begin{figure}[t]
\centering
\includegraphics[scale=.25]{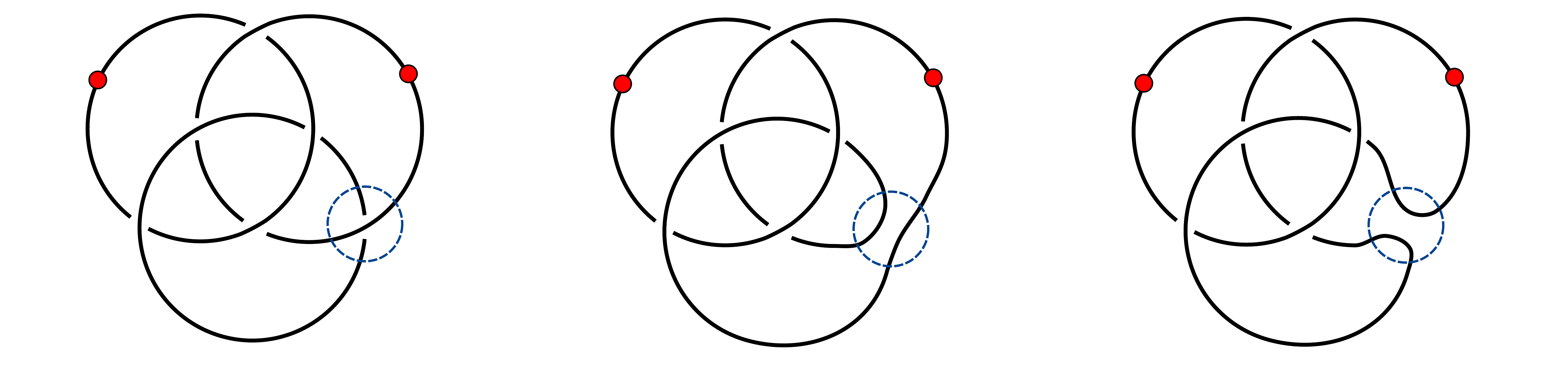}
\caption{Here we illustrate why L6n1 with the depicted non-trivial two-fold marking data (the left-most picture) is TQA. The middle picture is an unlink with non-trivial two-fold marking data, and thus is split into two odd-marked links. The right-most picture is a two-fold marked L4a1, which is alternating and hence TQA for any two-fold marking. Thus our two-fold marked L6n1 is TQA. Similar reasoning shows that L6n1 with the two other non-trivial two-fold markings are TQA.}\label{fig:L6n1isinQ2}
\end{figure}

\noindent The case in which $\omega$ is trivial and $L$ is QA was done in \cite{scaduto}. We note that when $\omega$ is non-trivial, $I^\#(\Sigma(L),\omega)$ is a priori only {\emph{relatively}} $\Z_4$-graded (and absolutely $\Z_2$-graded). To fix the $\Z_4$-grading one can choose an orientation of $L$. However, in the Theorem above it is evident that when $(L,\omega)$ is TQA and $\omega$ is non-trivial all $\Z_4$-gradings that lift the $\Z_2$-grading are equivalent. \\

\vspace{.35cm}

\noindent {\textbf{A spectral sequence for two-fold marked links.}} When $(L,\omega)$ is not TQA, we can still obtain quantitative results about $I^\#(\Sigma(L),\omega)$. We restrict to the case of mod 2 coefficients, and this will be implicit in all of the notation to follow. We first suggest:\\

\begin{conjecture}\label{conj:1}
There is a spectral sequence with $\Z_2$ coefficients with starting page $\text{{\emph{Kh}}}(L,\omega)$ converging to {\emph{$I^\#(\Sigma(L^\dagger),\omega)$}}, where {\emph{$L^\dagger$}} is the mirror of $L$.\\
\end{conjecture}

\noindent We will prove a less natural result which nonetheless gives rank bounds, and also provides the evidence for the conjecture. Given a two-fold marked diagram $(D,\check{\omega})$ representing $(L,\omega)$, we define $\text{Hd}(D,\check{\omega})$, the {\emph{dotted diagram homology of $(D,\check{\omega})$}. This is not an invariant of $(L,\omega)$, but arises naturally as a starting-page for a spectral sequence converging to $\text{Kh}(L,\omega)$. We will describe a spectral sequence with $E^2$-page $\text{Hd}(D,\check{\omega})$ that converges to $I^\#(\Sigma(L^\dagger),\omega)$. The result is stated below as Theorem \ref{thm:ss}, and we will shortly present an example.

\begin{figure}[t]
\centering
\includegraphics[scale=.40]{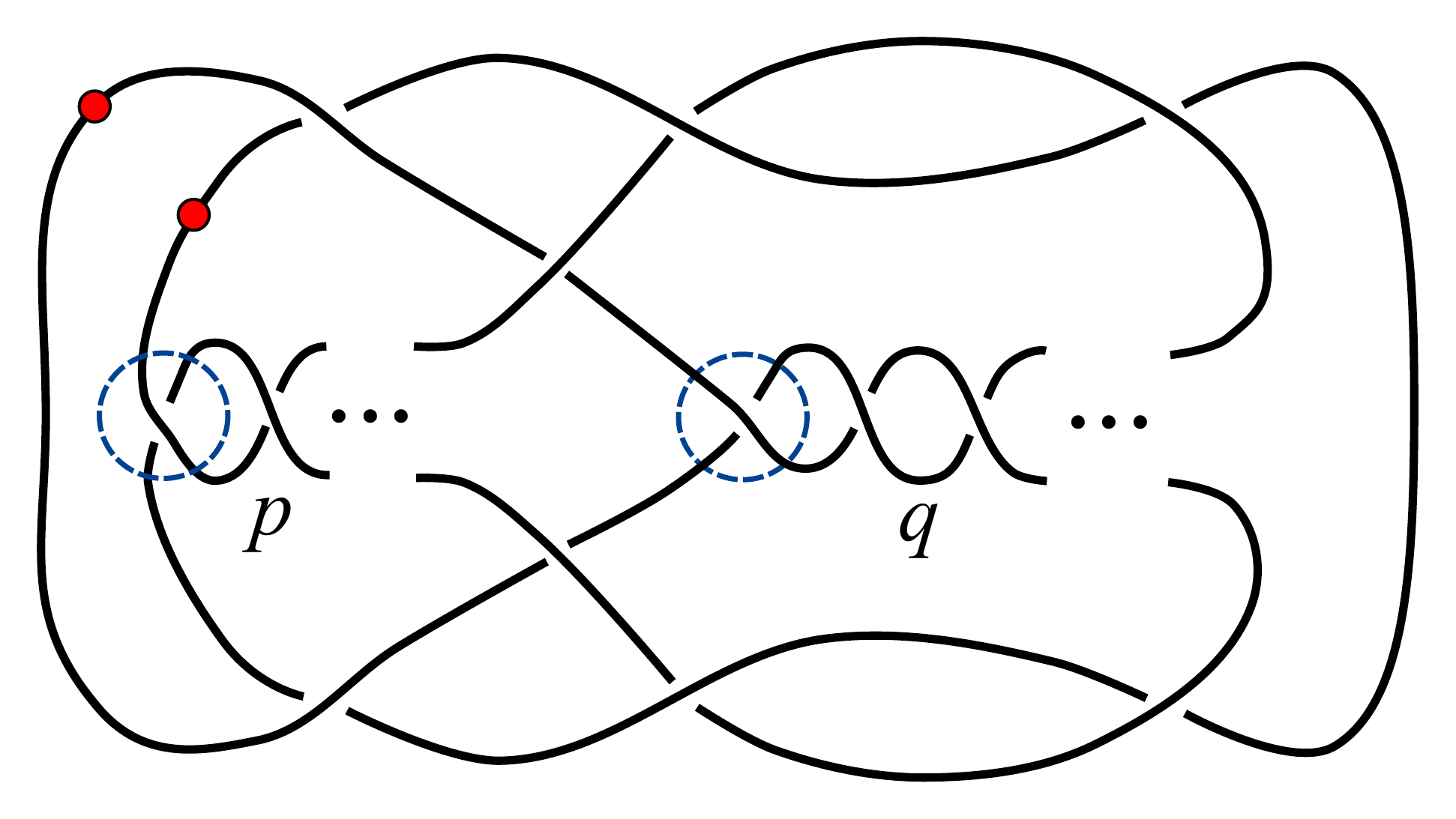}
\caption{The Kanenobu knot $K_{p,q}$. There are $p$ and $q$ many half-twists where indicated. $K_{0,0}$ is a sum of two $4_1$ knots, which is alternating. Also, $\det(K_{p,q})=25$ for all $p,q$. Resolving the left (resp. right) encircled crossing in two ways, we obtain an unlink split into odd-marked unknots, and $K_{p\pm 1,q}$ (resp. $K_{p,q\pm 1}$), where the sign depends on the sign of $p$ (resp. $q$). Thus $K_{p,q}$ is TQA by induction.}\label{fig:kanenobu}
\end{figure}

There is also a spectral sequence from $\text{Hd}(D,\check{\omega})$ converging to $\kh(L,\omega)$. In all cases known to the authors, for any $(L,\omega)$ there is a diagram $(D,\check{\omega})$ such that $\text{Hd}(D,\check{\omega})=\kh(L,\omega)$, and we speculate that this can always be arranged. We also mention that for any diagram $(D,\check{\omega})$ of $(L,\omega)$, the group $\kh(L,\omega)$ inherits from Theorem \ref{thm:ss} a $\Z_4$-grading $\delta^\#$, which can be expressed in terms of the $\delta$ grading and the homological grading $h$ on $\text{Hd}(D,\check{\omega})$, as well as a term $\theta\in \Z_2$ which is easily computed from the dotted oriented resolution of $(D,\check{\omega})$. We write

\begin{equation}
    \kh(L,\omega)_{\delta^\#}, \qquad \delta^\# := 2h - \delta + 2\theta \;\;\text{(mod 4)}\label{eq:z4group}
\end{equation}
\vspace{.01cm}

\noindent for the resulting $\Z_4$-graded group. Note that the $\delta^\#$ grading reduces mod 2 to the $\Z_2$-grading induced by $\delta$. For all $(L,\omega)$ in which $L$ is prime, has 10 crossings or less, and is more than one component, we were able to choose $\check{\omega}$ such that $\text{Hd}(D,\check{\omega})=\kh(L,\omega)$, and thus verified that there is a $\Z_4$-graded spectral sequence from (\ref{eq:z4group}) converging to $I^\#(\Sigma(L^\dagger),\omega)$. The computations of (\ref{eq:z4group}) are listed in Table \ref{table}. It was not clear to the authors how the grading $\delta^\#$ could be defined on $\kh(L,\omega)$ without reference to $\text{Hd}(D,\check{\omega})$, or if it is an invariant of $(L,\omega)$.

The simplest non-trivial example (with $\omega$ non-trivial) is the following. The first prime link $L$ that has a non-trivial two-fold marking $\omega$ for which $(L,\omega)$ is not TQA is the 3-component link L8n3. In fact, two of the three non-trivial two-fold markings for L8n3 are TQA, and the one that is not TQA is depicted in Figure \ref{fig:L8n3}, which we henceforth write as $(L,\omega)$. Theorem \ref{thm:ss} implies:
\[
	\text{dim}_{\Z_2} I^\#(\Sigma(L^\dagger),\omega)_{(i)} \leq 3, \quad i\in \{0,2\}\subset \Z_4,
\]
\[
	\text{dim}_{\Z_2} I^\#(\Sigma(L^\dagger),\omega)_{(i)} \leq 1, \quad i\in \{1,3\}\subset \Z_4.
\]
\vspace{-.20cm}

\noindent We note that the determinant of L8n3 is 4. For these inequalities, a priori one needs to fix an absolute $\Z_4$-grading, which can be done by orienting the link $L$. However, it is evident in this case that the choice does not matter. This seems to be a general phenomenon: in all of the cases known to the authors, the dotted diagram homology of Theorem \ref{thm:ss} has a 2-periodic $\Z_4$-grading.

Finally, we remark on Conjecture \ref{conj:1}. This conjecture does not follow immediately from the usual link surgeries method of \cite{ozsz} because the unperturbed framed instanton chain complex vanishes for the double branched cover of an unlink with a non-trivial $SO(3)$-bundle. One might like to use holonomy perturbations for each such branched cover in a way compatible with the structure of the spectral sequence, so that each resulting chain complex looks like the instanton homology of a branched cover with trivial bundle, but with non-trivial differential. One should then use a filtration on the sum of the perturbed complexes to yield Conjecture \ref{conj:1}. Adapting the filtration results of \cite{kmfiltrations} to framed instanton homology would be progress towards this approach.\\

\vspace{.65cm}

\noindent {\textbf{Further discussion.}} The genesis of spectral sequences from Khovanov homology to Floer homology is the work of Ozsv\'{a}th and Szab\'{o}, in the context of their Heegaard Floer homology \cite{os}. An analogue spectral sequence in Seiberg-Witten monopole Floer homology was constructed by Bloom \cite{bloom}. Each of these is with $\Z_2$ coefficients. These spectral sequences, with Thm. \ref{prop:tqa}, imply:\\

\begin{corollary}
    If a link $L$ is TQA, then the Heegaard-Floer homology $\widehat{HF}(\Sigma(L))$ and monopole Floer homology $\widetilde{HM}(\Sigma(L))$ with $\Z_2$ coefficients are of rank $\det(L)$.\\
\end{corollary}

\noindent We expect that this result holds over $\Z$, with the help of twisted exact triangles. In the context of instanton homology, Kronheimer and Mrowka \cite{kmunknot} 
proved that reduced Khovanov homology with $\Z$ coefficients detects the unknot, by constructing a spectral sequence from reduced Khovanov homology to the $\Z_4$-graded singular instanton homology $I^\natural(L)$. Their spectral sequence implies\\

\begin{corollary}
    If $L$ is TQA, then $I^\natural(L)$ with $\Z_2$ coefficients is of rank $\det(L)$.\\
\end{corollary}

\noindent We do not know if this result should hold over $\Z$.

Using ideas related to the aforementioned spectral sequences, see also \cite{kmfiltrations}, Daemi \cite{daemi} showed how abelian instantons induce filtrations on reduced odd Khovanov homology. The first author \cite{scaduto} constructed a spectral sequence from the reduced odd Khovanov homology converging to the framed instanton homology of the double branched cover.
The current article is the result of attempting to extend the results of \cite{scaduto} to handle non-trivial bundles.

We expect that our restriction to $\Z_2$ coefficients in Theorem \ref{thm:ss} can be removed. The twisted Khovanov theory with signs due to Manion \cite{manion} should be relevant here. We also mention that our constructions are reminiscent of the spanning tree models that have appeared in various link and Floer homology theories \cite{greenespanning,baldwinlevine,wehrli,ckspanning}.

Finally, we mention that the constructions in this paper bear similarity to Kronheimer and Mrowka's $SO(3)$ instanton homology for webs \cite{kmtait}. By drawing explicit arcs for $\omega$ to represent the Poincar\'{e} dual of the second Stiefel-Whitney class of the associated $SO(3)$ bundle over $\Sigma(L)$, we obtain a web $L\cup \omega$. The exact triangles for the instanton homology of such webs in \cite{kmtriangle} are formally similar to ones we use.\\

\vspace{.65cm}

\noindent {\textbf{Outline.}} We define $\kh(L,\omega)$ and $\text{Hd}(D,\check{\omega})$ in Section \ref{sec:twisted}. In Section \ref{sec:framed}, we review framed instanton homology for double branched covers, $I^\#(\Sigma(L),\omega)$, and in Section \ref{sec:gradings} we study the degrees of cobordism maps. The spectral sequence from $\text{Hd}(D,\check{\omega})$ converging to $I^\#(\Sigma(L^\dagger),\omega)$ is presented in Section \ref{sec:ss}. In Section \ref{sec:tqa} we prove Thm. \ref{thm:tqa}. As noted in Section \ref{sec:twisted}, the proof of Thm. \ref{prop:tqa} is similar. In Section \ref{sec:more}, we find many TQA Montesinos knots. In Section \ref{sec:computer}, we list computer calculations of twisted Khovanov homology, and discuss the resulting rank inequalities for $I^\#(\Sigma(L),\omega)$. The reader only interested in Khovanov homology can read Sections \ref{sec:twisted}, \ref{sec:more} and \ref{sec:computer}.\\

\vspace{.30cm}

\noindent {\textbf{Acknowledgements.}} The authors would like to thank Aliakbar Daemi for helpful comments and suggestions. We also thank John Baldwin, Francis Bonahon, Josh Greene, Andy Manion and Ciprian Manolescu for helpful discussions. The first author was supported by NSF grant DMS-1503100.

\vspace{.85cm}

\newpage

\section{Twisted Khovanov homology and dotted diagrams}\label{sec:twisted}
In this section we describe our construction of twisted Khovanov homology $\text{Kh}(L,\omega)$, which is an invariant of the (oriented) two-fold marked link $(L,\omega)$. We work with $\F=\Z_2$ coefficients throughout. This invariant is a modification of Roberts' totally twisted Khovanov homology \cite{roberts} in the spirit of Jaeger \cite{jaeger}. In particular, we use the viewpoint of dots on arcs as in \cite{jaeger} instead of regions, which was used in \cite{roberts}. We also mention that Baldwin, Levine and Sarkar recently introduced Khovanov homology for pointed links, and they relate their construction to the aforementioned ones \cite[Remark 2.8]{bls}. Our construction is obtained from these by working over $\F$ and setting certain polynomial variables equal to $0$ or $1$, as determined by the dots. We warn the reader that the conventions we use do not uniformly agree with any of these references.

Let $(D,\check{\omega})$ be a two-fold marked diagram compatible with $(L,\omega)$. Consider the cube of $2^n$ (complete) resolutions of $D$, where $n$ is the number of crossings in $D$. For a given $u\in \{0,1\}^n$ and corresponding resolution diagram $D_u$ of disjoint circles, we obtain two-fold marking data $\check{\omega}_u$ for $D_u$ by carrying the dots along in the natural way, and by counting dots modulo 2.

Now we form the twisted Khovanov complex $C(D,\check{\omega})$. For each resolution vertex $u\in\{0,1\}^n$ the resolution diagram $D_u$ is a disjoint union of planar circles which we label $a_1,\ldots,a_k$. Define the $\F$-exterior algebra $\Lambda_u=\Lambda(a_1,\ldots,a_k)$ with generators given by the circles. Consider the subalgebra $C_u \subset \Lambda_u$ generated by the kernel of the augmentation map $\Lambda_u\longrightarrow \F$ which sends each $a_j$ to $1$.  The vector space $C(D,\check{\omega})$ is then given by the direct sum of the $C_u$ over all vertices $u\in\{0,1\}^n$. 
The differential $d$ is given as a sum of ``horizontal'' and ``vertical'' differentials:
\[
	d = d_v + d_h.
\]
The differential $d_h$ is the usual Khovanov differential, see Figure \ref{fig:mergesandsplits}. The vertical differential $d_v$ is a differential on each $C_u$, and is given as follows. Let the resolution diagram $D_u$ have circles $a_1,\ldots,a_k$, and let $\check{\omega}_j\in \{0,1\}$ record the number of dots modulo 2 on circle $a_j$. Then
\begin{equation}
	d_v(\;\cdot\;) =\sum_{j=1}^k\check{\omega}_j a_j\wedge(\; \cdot\;).\label{eq:vert}
\end{equation}
We equip our chain complex with a $\Z$-grading, called $\delta$, which is linear combination of a quantum grading $(q)$ and homological grading $(h)$. To define these gradings we orient $L$, which induces an orientation of $D$. For an element $x=a_1\wedge \cdots \wedge a_i \in C_u$, we define:
\[
	q(x) = k-1-2i  + |u|_1 + n_+ - 2n_-, \qquad h(x) = |u|_1 - n_-.
\]
Here $k$ is the number of circles in $D_u$, $n_{\pm}$ is the number of $\pm $-crossings in the diagram $D$, and $|u|_1$ is the $L^1$-norm of the vertex $u\in \{0,1\}^n$. Then our $\delta$-grading is defined to be
\begin{equation}
	\delta = \frac{1}{2}q - h -\frac{1}{2}\sigma+\frac{3}{2}\nu\label{eq:deltagrad}
\end{equation}
where $\sigma$ and $\nu$ are the signature and nullity of $L$. The nullity may be defined as $b_1(\Sigma(L))$. Our convention, as in \cite{bloom}, is that the signature of the right-handed trefoil is +2, which is minus the signature of a Seifert matrix.  \\

\begin{defn}
	The {\emph{twisted Khovanov homology of $(L,\omega)$}} is defined to be the homology of the $\delta$-graded chain complex $(C(D,\check{\omega}),d)$. It is denoted $\text{{\emph{Kh}}}(L,\omega)$.\\
\end{defn}

\noindent When $\omega$ consists of all zeros, i.e. the two-fold marking data is ``trivial,'' we write $\omega=\vec{0}$. Note that for a {\emph{knot}} there is only the trivial two-fold marking data $\vec{0}$. We list some basic properties of $\text{Kh}(L,\omega)$.\\

\begin{itemize}
\item The $\delta$-graded homology $\text{Kh}(L,\omega)$ is an invariant of the oriented link $L$ and choice of two-fold marking data $\omega$ (cf. \cite[Thm. 2.10]{roberts}, \cite[Thm. 3.1]{jaeger}).
\item The $\delta$-graded twisted homology $\text{Kh}(L,\vec{0})$ is the same as the (reduced) $\delta$-graded Khovanov homology $\overline{\text{Kh}}(L;\F)$.
\item If $(L,\omega)$ is TQA, then $\text{Kh}(L,\omega)$ is supported in grading $\delta =0$, and has $\F$-dimension $\text{det}(L)$ (In the QA case, see \cite[Thm. 2.18]{roberts}).
\item If $(L,\omega)$ splits into odd-marked links, then $\kh(L,\omega)=0$.\\
\end{itemize}

\noindent The third item is of course Theorem \ref{prop:tqa}. Its proof is similar to the proof of Theorem \ref{thm:tqa}, which is presented in Section \ref{sec:tqa}, and for this reason, we only indicate here how it differs.
First, it was mentioned in the introduction that the proof of Theorem \ref{prop:tqa} is similar to that of the untwisted, quasi-alterating case, cf. \cite{mo}. More specifically, one uses the twisted exact triangle\\
\begin{equation}
\begin{tikzcd}
                \cdots \;\;\kh(L_1,\omega_1) \arrow{r} &  \kh(L,\omega) \arrow{r}     &  \kh(L_0,\omega_0)              \arrow{r}                 &  \kh(L_1,\omega_1)\;\; \cdots 
\end{tikzcd}\label{eq:skeintriangle}
\end{equation}
\vspace{.0cm}

\noindent in which the three two-fold marked links are related by dotted resolution diagrams in the usual way. This exact sequence follows in the same way as does the untwisted version. Second, this exact sequence behaves the same way with respect to the $\delta$-gradings as in \cite{mo}. Finally, suppose that $\det(L_1)=0$ and $\kh(L_1,\omega_1)=0$. Then (\ref{eq:skeintriangle}) induces an isomorphism from $\kh(L,\omega)$ to $\kh(L_0,\omega_0)$ which respects the $\delta$-gradings. (Indeed, \cite[Lemma 2.1]{mo} holds only assuming $\det(L_v)>0$ and $\det(L_h)=0$, from whence this claim follows from \cite[Prop. 2.3]{mo}.) From these remarks and the argument for Theorem \ref{thm:tqa}, Theorem \ref{prop:tqa} follows.\\
\vspace{.50cm}

\begin{figure}[t]
\centering
\includegraphics[scale=.33]{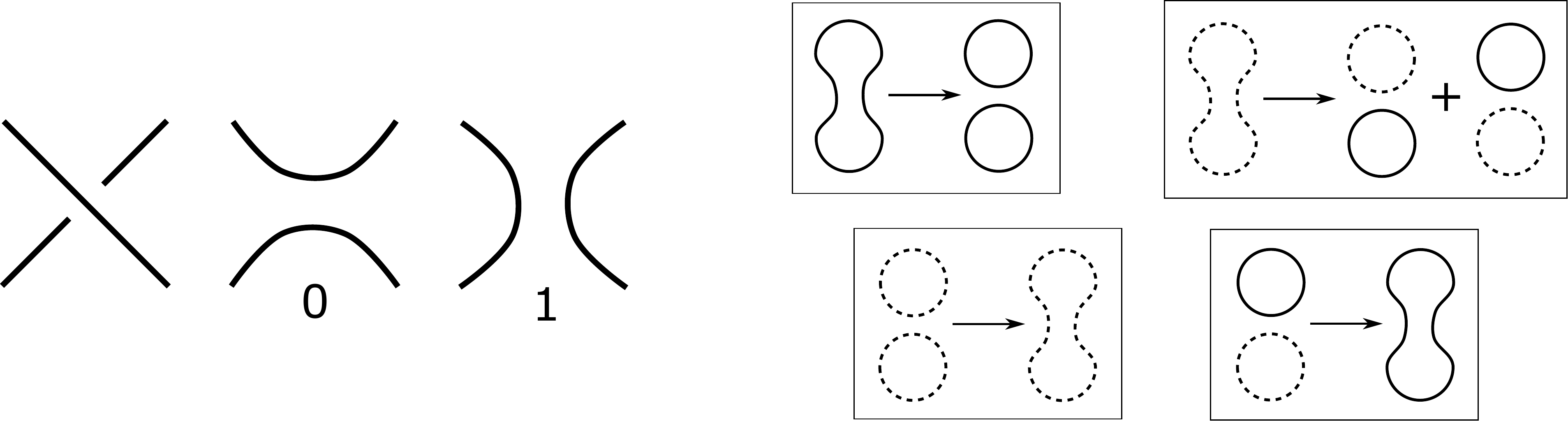}
\caption{On the left, we provide the convention for 0- and 1-resolutions of crossings. On the right, we illustrate the (horizontal, unreduced) Khovanov differential for splits and merges. For a given complete resolution with circles $a_1,a_2,\ldots$, an element $a_1\wedge \cdots \wedge a_i$ corresponds to the picture in which $a_1,\ldots,a_i$ are solid and all other circles are dotted.}\label{fig:mergesandsplits}
\end{figure}

\noindent {\bf{The double complex.}} The horizontal and vertical differentials $d_h$ and $d_v$ commute with one another, whence $C(D,\check{\omega})$ inherits the structure of a double complex. In particular, we can first take the vertical homology, $H_\ast(C(D,\check{\omega}),d_v)$. This has a particularly simple form. Referring to (\ref{eq:vert}), we observe that on $C_u$ the differential $d_v$ is an isomorphism or zero, depending on whether or not $\check{\omega}_u$ is non-trivial. Thus the vertical homology is the same vector space as $C(D,\check{\omega})$ but with summands $C_u$ thrown out if there are an odd number of dots on any circle in $D_u$. An example is depicted in Figure \ref{fig:L4a1}. We make the following definition:\\

\begin{defn}
	Denote by $\text{{\emph{Hd}}}(D,\check{\omega})$ the homology of the vertical homology $H_\ast(C(D,\check{\omega}),d_v)$ with respect to $d_h$. We call  $\text{{\emph{Hd}}}(D,\check{\omega})$ the {\emph{dotted-diagram homology of $(D,\check{\omega})$}}.\\
\end{defn}

\noindent The dotted-diagram homology is not an invariant of $(L,\omega)$. It is invariant under any version of dotted Reidemeister I moves,  but not dotted Reidemeister II or Reidemeister III moves. See Figure \ref{fig:noninvariance}. We do note, however, that certain dotted moves, depicted in Figure \ref{fig:noninvariance}, do not change the complex $C(D,\check{\omega})$ at all, and so $\text{Hd}(D,\check{\omega})$ is invariant under such moves.\\

\begin{remark}
	Despite not being an invariant for $(L,\omega)$, we have given $\text{{\emph{Hd}}}(D,\check{\omega})$ a name because it is the starting page for a naturally arising spectral sequence converging to framed instanton homology of the branched double cover, as we discuss in the next section.\\
\end{remark}

\begin{figure}[t]
\centering
\includegraphics[scale=.17]{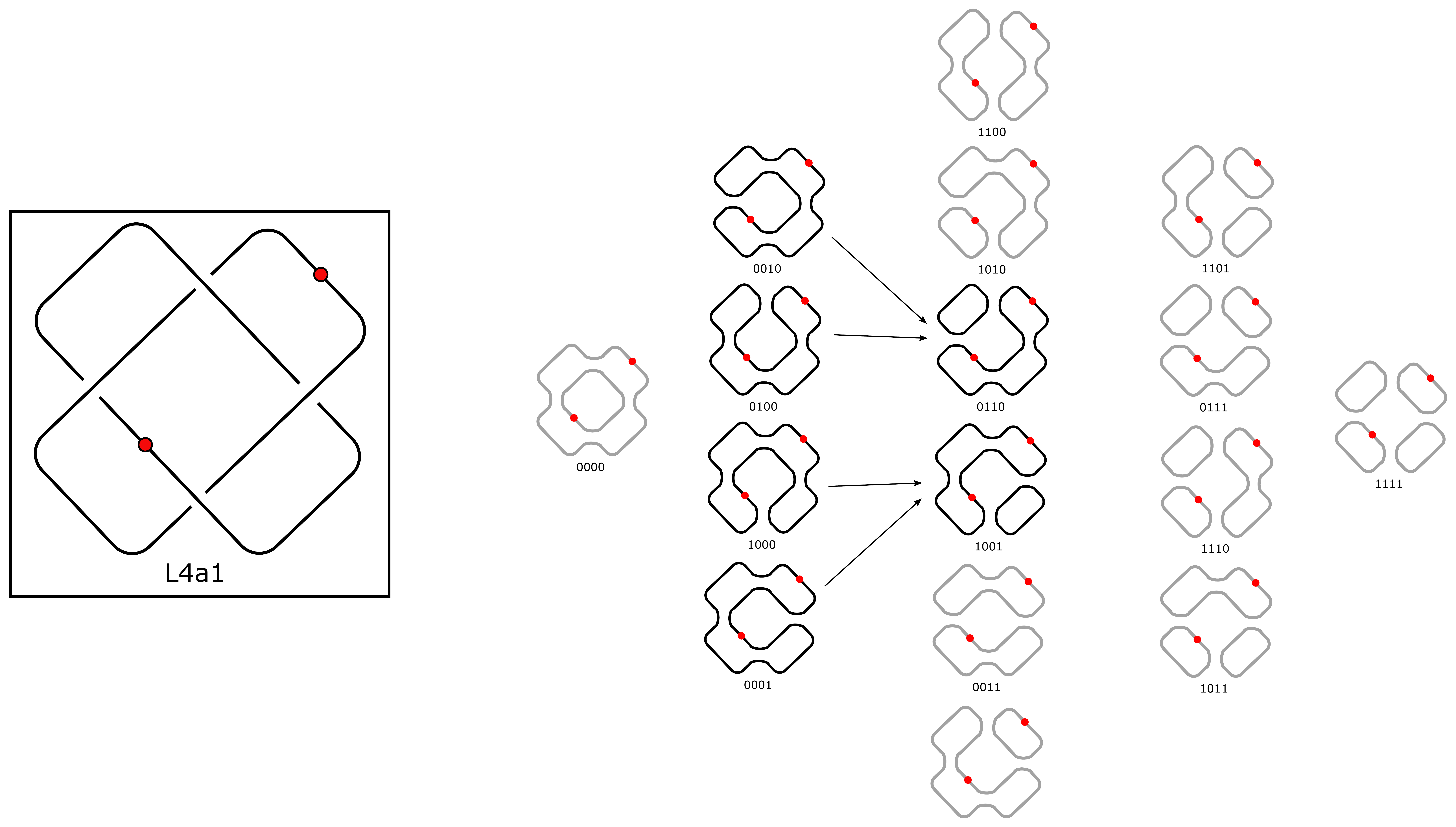}
\caption{The dotted-diagram complex, i.e. the vertical homology $H_\ast(C(D,\check{\omega}),d_v)$ with differential induced by $d_h$, for the two-component link L4a1 with non-trivial two-fold marking data. The complex is the usual reduced Khovanov complex, but with chain groups at resolutions with any circle supporting an odd number of dots omitted; these resolutions are faded in the illustration. This complex computes the twisted Khovanov homology, which is the same as the Khovanov homology, since L4a1 is alternating. The homology is rank $4=\text{det}$(L4a1), and is supported in $\delta$-grading 0.} \label{fig:L4a1}
\end{figure}

From the double complex structure of $(C(D,\check{\omega}),d)$, upon taking the vertical and horizontal homologies in different orders, we obtain the following two facts:\\

\begin{itemize}
\item There is a spectral sequence with $E^2$-page $\text{Hd}(D,\check{\omega})$ converging to $\text{Kh}(L,\omega)$.
\item There is a spectral sequence with $E^1$-page $\overline{\text{Kh}}(L;\F)$ converging to $\text{Kh}(L,\omega)$.\\
\end{itemize}

\noindent These are both $\delta$-graded. From the above discussion, we know that the second spectral sequence collapses when $L$ is TQA. For the first spectral sequence, we have the following observation. Suppose that for all $u,w\in\{0,1\}^n$ with $\check{\omega}_u,\check{\omega}_w$ trivial, there is no $v\in\{0,1\}^n$ with $\check{\omega}_v$ non-trivial such that $u \leq v \leq w$. Then it follows from the double complex filtration that the dotted diagram homology computes the twisted Khovanov homology. This happens also if $\text{Hd}(D,\check{\omega})$ is supported in two adjacent homological gradings, which is seen to hold in Figures \ref{fig:L4a1} and \ref{fig:L6n1comp}. \\

\begin{figure}[t]
\centering
\includegraphics[scale=.33]{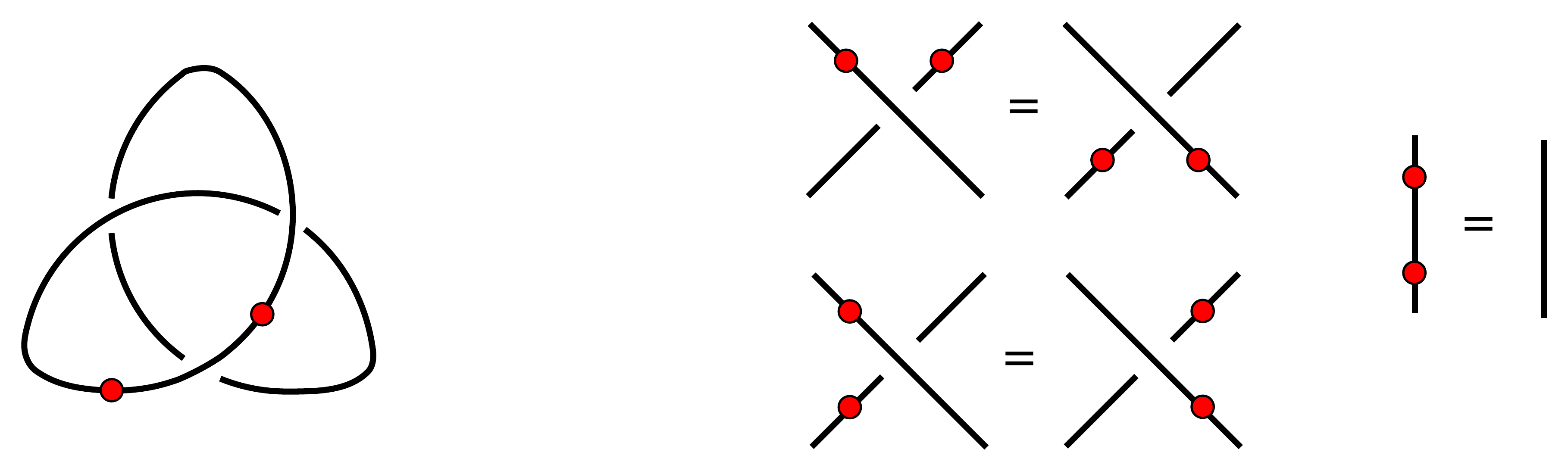}
\caption{On the left is a diagram $D$ of the unknot with two-fold marking data $\check{\omega}$ which is equivalent to the trivial two-fold marking data. However, in this case $\text{Hd}(D,\check{\omega})$ has rank $3$ instead of $1$. Since this diagram can be transformed into the unknot via either a (dotted) Reidemeister II or III move along with some Reidemeister I moves (all of which are valid), we cannot have invariance for either such dotted II or III moves. On the right we list some dotted moves that leave $C(D,\check{\omega})$ unchanged, and thus also $\text{Hd}(D,\check{\omega})$.}\label{fig:noninvariance}
\end{figure}

\section{Framed instanton homology for branched double covers}\label{sec:framed}

The framed instanton homology $I^\#(Y,\omega)$ is a $\Z_2$-graded abelian group defined for a closed, oriented and connected 3-manifold $Y$ and a choice of an embedded, oriented 1-manifold $\omega\subset Y$. However, until Section \ref{sec:tqa} we will assume the use of $\F$-coefficients. The isomorphism class of $I^\#(Y,\omega)$ only depends on $[\omega]\in H_1(Y;\Z_2)$. It is a Floer-homology group for the Chern-Simons functional defined on the framed configuration space of an $SO(3)$-bundle $E$ over $Y$ which has $w_2(E)$ Poincar\'{e} dual to $[\omega]$. The group $I^\#(Y,\omega)$ is 4-consecutive gradings of Floer's relatively $\Z_8$-graded (and absolutely $\Z_2$-graded) instanton homology for a non-trivial admissible bundle over $Y\# T^3$ which restricts to $E$ over $Y$ and a non-trivial bundle over $T^3$. Technically, we should keep track of the basepoint in $Y$ where the 3-torus is attached, but we will keep this hidden.

Let $L$ be a link in $S^3$ and $\Sigma(L)$ its branched double cover. With the correspondence of (\ref{eq:bundlemarking}) understood, we will use the notation $\omega$ for both the two-fold marking data on $L$, and a collection of embedded loops in $\Sigma(L)$ Poincar\'{e} dual to the second Stiefel-Whitney class of the associated $SO(3)$-bundle over $\Sigma(L)$. First we make a convenient definition.\\

\begin{defn}
	A two-fold marked diagram $(D,\check{\omega})$ is a {\emph{dotted surgery diagram}} if, up to the dotted moves in Figure \ref{fig:noninvariance}, it may be represented by a dotted picture of $D$ in which all dots are located near crossings, and near each crossing, there are either 0 or 2 dots. When there are 2 dots at a crossing we require that they are on adjacent arcs, not diagonal from each other. See Figure \ref{fig:surgerydiagram}.\\
\end{defn}

\noindent It seems likely that all two-fold marked diagrams $(D,\check{\omega})$ which are connected} are dotted surgery diagrams. (Note that a two-fold marked unlink with non-trivial marking data is not a dotted surgery diagram.)  In any case, it is easy to see that for any two-fold marked link $(L,\omega)$ there exists some dotted surgery diagram $(D,\check{\omega})$ representing it. We write $L^\dagger$ for the mirror of a link $L$, and similarly for diagrams.\\

\begin{prop} \label{prop:ss} Let $L$ be a link with two-fold marking data $\omega$ and let $D$ be a diagram of $L$ with two-fold marking data $\check{\omega}$ compatible with $\omega$.	Suppose $(D,\check{\omega})$ is a dotted surgery diagram. Then there is a spectral sequence with $E^2$-page isomorphic to the $\delta$ mod 2 graded dotted diagram homology $\text{{\emph{Hd}}}(D,\check{\omega})$ that converges to the mod 2 graded framed instanton homology {\emph{$I^\#(\Sigma(L^\dagger),\omega)$}}.\\
\end{prop}

\vspace{.20cm}

\begin{proof}
At a distinguished crossing of $D$, write $D_0$ and $D_1$ for the diagrams resulting from only performing a $0$- or $1$-resolution at that crossing, respectively. For $i=0,1$ we then write $\check{\omega}_i$ for the induced two-fold marking data on the diagram $D_i$. The relevant exact triangle is given as follows:\\
\begin{equation}
\begin{tikzcd}
                \cdots \;\; I^\#(\Sigma(D^\dagger),\check{\omega}\cup \mu) \arrow{r} &  I^\#(\Sigma(D_0^\dagger),\check{\omega}_0) \arrow{r}     &  I^\#(\Sigma(D_1^\dagger),\check{\omega}_1)              \arrow{r}                 &  I^\#(\Sigma(D^\dagger),\check{\omega}\cup\mu)\;\; \cdots  \label{eq:ses}
\end{tikzcd}
\end{equation}
\vspace{.15cm}

\noindent Here $\mu$ is an embedded loop in $\Sigma(D^\dagger)$ (resp. $\Sigma(D_0^\dagger)$, $\Sigma(D_1^\dagger)$) which is the lift of a small arc placed between the two strands near the relevant crossing in $D$ (resp. $D_0$, $D_1$). We apply this exact triangle to each crossing. More precisely, we apply the link surgeries spectral sequence of \cite{scaduto} with these choices. The result is a spectral sequence with $E^2$-page given by 
\[
	\bigoplus_{u\in \{0,1\}^n} I^\#(\Sigma(D_u^\dagger),\check{\omega}_u).
\]
If $\check{\omega}_u$ is non-trivial, then the associated $SO(3)$-bundle over $\Sigma(D_u^\dagger)$ is non-trivial. The framed instanton homology of $S^1\times S^2$ with a non-trivial bundle is zero, and this property persists to connect sums of $S^1\times S^2$. Further, the differential on the non-zero summands of the $E^2$-page matches the reduced Khovanov differential, cf. \cite{scaduto}. This identifies the $E^2$-page with $\text{Hd}(D,\check{\omega})$. Next, the spectral sequence converges to $I^\#(\Sigma(L^\dagger),\omega\cup\mu)$, where $\mu$ is the union of lifted arcs, one for each crossing. However, $\mu$ is mod 2 null-homologous. (In fact, an explicit bounding surface for $\mu$ is constructed in \cite{scaduto}. See also the proof Theorem \ref{thm:tqa} in Section \ref{sec:tqa}.) Finally, the mod 2 gradings are just as in \cite{scaduto}, and the bundles play no role in this grading.
\end{proof}

\begin{figure}[t]
\centering
\includegraphics[scale=.35]{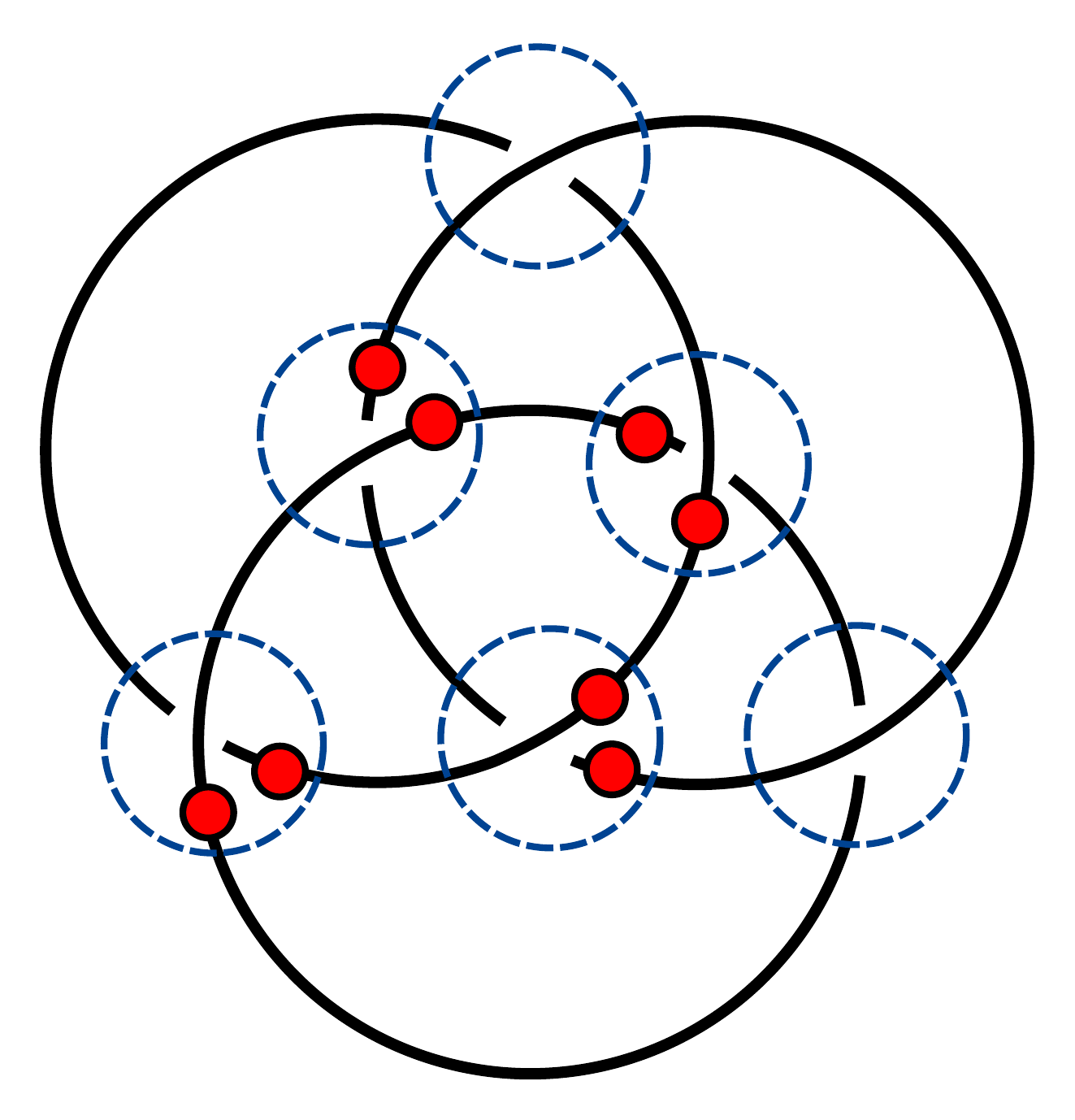}
\caption{A two-fold marked diagram for L6n1 which is a dotted surgery diagram. All dots are located near crossings. Near each crossing, there are either no dots, or two dots on adjacent arcs.}\label{fig:surgerydiagram}
\end{figure}

\vspace{.55cm}

\section{Spin structures and mod 4 gradings}\label{sec:gradings}

In general, the $\Z_2$-grading of $I^\#(Y,\omega)$ is covered by a relative $\Z_4$-grading. Write $\text{Spin}(Y)$ for the set of spin structures on $Y$ up to equivalence. Following Fr{\o}yshov \cite{froyshov}, a choice of spin structure $\mathfrak{s}\in \text{Spin}(Y)$ determines an {\emph{absolute}} $\Z_4$-grading $\text{gr}[\mathfrak{s}]$ lifting the $\Z_2$-grading. We write $\text{gr}_Y[\mathfrak{s}]$ when the dependence on $Y$ is to be emphasized. Further, Fr{\o}yshov points out that for $\mathfrak{s}_1,\mathfrak{s}_2\in \text{Spin}(Y)$,
\begin{equation}
	\text{gr}[\mathfrak{s}_2] - \text{gr}[\mathfrak{s}_1] \equiv 2\cdot \langle \mathfrak{s}_1 - \mathfrak{s}_2,[\omega]\rangle \mod 4.\label{eq:gradings}
\end{equation}
Here we recall that the difference of two spin structures is an element of $H^1(Y;\Z_2)$, and $\langle\cdot,\cdot\rangle$ is the natural pairing between $H^1(Y;\Z_2)$ and $H_1(Y;\Z_2)$. Now we return to double branched covers. Let $L$ be a link and write $\text{Or}(L)$ for the set of orientations of $L$. Given an orientation $o\in \text{Or}(L)$, we denote by $-o$ the orientation obtained from simultaneously reversing the orientations for each component of $L$ given by $o$. Turaev \cite{turaev} gives the following canonical bijection:\\
\[
\text{Spin}(\Sigma(L))
 \;\; \overset{1:1}\longleftrightarrow \;\; \text{Or}(L)/\pm
\]
Turaev refers to the equivalence classes $\pm o$ as {\emph{quasi-orientations}}.\\

\vspace{.15cm}

\noindent {\textbf{Gradings for framed instanton homology.}} Suppose $\mathfrak{a}$ is a generator for the chain complex computing $I^\#(Y,\omega)$. This means that $\mathfrak{a}$ is a critical point for a perturbed Chern-Simons functional defined on the relevant $SO(3)$-bundle over $Y\# T^3$. Let $X$ be a spin 4-manifold bounding $(Y,\mathfrak{s})$. Then $\text{gr}[\mathfrak{s}](\mathfrak{a})\in\Z_4$ is defined to be mod 4 congruent to
\[
	-\text{ind}(d_A^\ast \oplus d_A^+) -b_1(X) +b_+(X) -b_1(Y)
\]
up to the addition of a universal constant in $\Z_4$, which is chosen so that $I^\#(S^3)$ is supported in grading $0$. Here $d_A^\ast \oplus d_A^+$ is the ASD operator for a connection $A$, limiting to $\mathfrak{a}$, on an $SO(3)$ bundle over the boundary sum of $X$ and $T^2\times D^2$, with a cylindrical end attached. The operator is defined on suitable Sobolev completions, and its index is the expected dimension of the moduli space of finite-energy instantons limiting to $\mathfrak{a}$ for this cylindrical-end manifold with bundle. For a $(3+1)$-dimensional cobordism $X$ from $Y_1$ to $Y_2$ we define the following integer:

\[
	d(X):= -\frac{3}{2}(\chi(X) + \sigma(X)) +\frac{1}{2}(b_1(Y_2)-b_1(Y_1)) \in \Z.
\]
\vspace{.10cm}

\noindent Suppose further that $E$ is a bundle over $X$ that restricts to the bundle data determined by $\omega_i$ over $Y_i$ for $i=1,2$. We will encode the cobordism data $(X,E)$ in the more convenient, geometric form of the pair $(X,S)$, where $S$ is an embedded, unoriented surface in $X$ with $\partial S = \omega_1 \cup \omega_2$, and which represents the Poincar\'{e} dual of $w_2(E)$. For the following, we use the notation $\text{deg}[\mathfrak{s}_1,\mathfrak{s}_2]$ to record the mod 4 degree of a map $I^\#(Y_1,\omega_1)\longrightarrow I^\#(Y_2,\omega_2)$ in which for $i=1,2$ the vector space $I^\#(Y_i,\omega_i)$ is equipped with the $\Z_4$-grading $\text{gr}[\mathfrak{s}_i]$.\\

\begin{figure}[t]
\centering
\includegraphics[scale=.40]{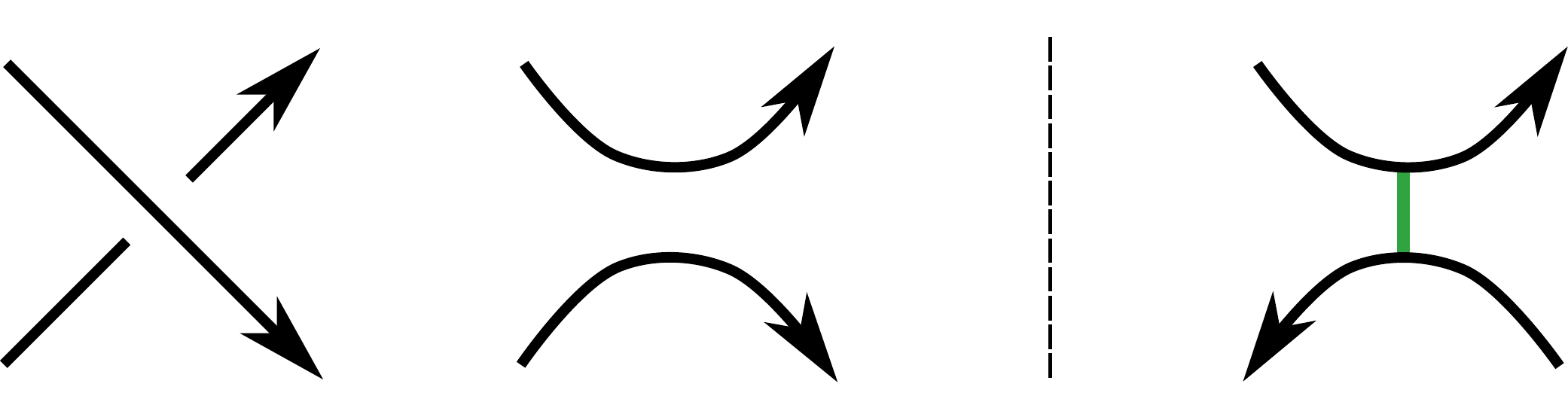}
\caption{On the left is depicted an oriented resolution. On the right, an arc connecting two oppositely directed segments, as appears in the definition of the canonical quasi-orientation.}\label{fig:orientedmoves}
\end{figure}

\vspace{.15cm}

\begin{prop} \label{prop:deg1}
Suppose $(X,S)$ is a cobordism from $(Y_1,\omega_1)$ to $(Y_2,\omega_2)$. If $X$ admits a spin structure restricting on the boundary to $\mathfrak{s}_i\in\text{{\emph{Spin}}}(Y_i)$ for $i=1,2$, then
\[
	\text{{\emph{deg}}}[\mathfrak{s}_1,\mathfrak{s}_2]\left(I^\#(X,S) \right) \equiv d(X) \mod 4.
\]
If $X$ is not neccesarily spin, but $\omega_1$ and $\omega_2$ are mod 2 nullhomologous, then
\[
	\text{{\emph{deg}}}[\mathfrak{s}_1,\mathfrak{s}_2]\left(I^\#(X,S) \right) \equiv  d(X) + 2(S\cdot S) \mod 4,
\]
and in this case the degree is independent of the choices $\mathfrak{s}_1,\mathfrak{s}_2$.\\
\end{prop}
\vspace{.25cm}
\noindent  The proof is a standard application of index additivity and some algebraic topology. The second statement is \cite[4.3.1]{scaduto} and the first follows similarly. We record one more useful special case:\\

\begin{prop}
	The effect on the degree of blowing up the interior of a cobordism by $(\overline{\mathbb{C}\mathbb{P}}^2, F)$ where $F$ is a sphere of self-intersection $-1$ is addition by 2:
\[
\text{{\emph{deg}}}[\mathfrak{s}_1,\mathfrak{s}_2]\left(I^\#(X\#\overline{\mathbb{C}\mathbb{P}}^2,S\cup F) \right)  \equiv \text{{\emph{deg}}}[\mathfrak{s}_1,\mathfrak{s}_2]\left(I^\#(X,S) \right)  + 2\mod 4.
\] \label{prop:deg}
\end{prop}

\noindent Note that in general, the $[\mathfrak{s}_1,\mathfrak{s}_2]$-degree of a cobordism map $(X,S)$ is mod 4 congruent to $d(X) + \Delta$, where this is the defining equation for $\Delta =\Delta[\mathfrak{s}_1,\mathfrak{s}_2](X,S) \in 2\cdot \Z_4$. In general, both $d$ and $\Delta$ are additive with respect to the composition of cobordism data. Proposition \ref{prop:deg1} says that $\Delta\equiv 0$ (mod 4) when $X$ is spin, and $\Delta\equiv  2(S\cdot S)$ (mod 4) when the boundary bundles are trivial. Proposition \ref{prop:deg} says that $\Delta$ changes by $2$ when blowing up, with non-trivial $SO(3)$-bundle. In fact, Prop. \ref{prop:deg} more or less follows from Prop. \ref{prop:deg1} by degree additivity.

When $Y$ is replaced with its orientation reversal, one takes the dual of the instanton homology with the original grading reflected and shifted by $b_1(Y)$. Precisely, we have
\[
    \text{gr}_{-Y}[\mathfrak{s}] \equiv b_1(Y) - \text{gr}_{Y}[\mathfrak{s}] \text{ (mod 4)},
\]
where $\mathfrak{s}\in\text{Spin}(Y)$, and we give the same name to the related spin structure on $-Y$. We will not need this, but we note that since $\Sigma(L^\dagger)=-\Sigma(L)$, the framed instanton homology of $\Sigma(L^\dagger)$ with field coefficients is easily related by the above to that of $\Sigma(L)$.

Finally, we remark that our cobordisms $X$ above should technically be equipped with a path from the basepoint of $Y_1$ to the basepoint of $Y_2$. However, just as in \cite{scaduto}, in our applications there is always an implicit choice for such a path, and as such we keep them hidden. Also, although we are working only with $\F=\Z_2$ coefficients, we remark that to work with signs, one must choose some auxiliary data, such as a homology orientation for $X$.\\

\vspace{.55cm}

\noindent {\textbf{Gradings for branched covers.}} Now let $L$ be a link with diagram $D$. Fix a crossing of $D$, and for $i=0,1$ let $D_i$ be the diagram obtained by performing an $i$-smoothing at the crossing, as in Figure \ref{fig:mergesandsplits}. Since $D_i$ can be viewed as zero-surgery on a certain framed knot in $\Sigma(L^\dagger)$, we have cobordisms $V_{i}:\Sigma(D^\dagger)\longrightarrow \Sigma(D_i^\dagger)$ which are two-handle additions to $\Sigma(D^\dagger)\times [0,1]$. Now choose an orientation $o$ of $L$ and write $\mathfrak{s}$ for the spin structure of $\Sigma(L^\dagger)$ determined by $o$. (We will use the notation $o$ for both the orientation of $L$ and the induced orientation on its mirror $L^\dagger$.) Now, there is a preferred resolution amongst $D_0$ and $D_1$, the oriented resolution, see Figure \ref{fig:orientedmoves}. If $D_i$ is the oriented resolution, then it inherits an orientation from $D$, and we write $\mathfrak{s}_i$ for the associated spin structure.\\

\begin{prop}[\hspace{-.02cm}{\cite[Prop. 1]{liscaowens}}]
	If $D_i$ is an oriented resolution of $D$, then the cobordism {\emph{$V_i:\Sigma(D^\dagger)\longrightarrow \Sigma(D_i^\dagger)$}} has a spin structure restricting on the ends to $\mathfrak{s}$ and $\mathfrak{s}_i$. \label{prop:blowup} \\
\end{prop}

\noindent More generally, let $u\in \{0,1\}^n$ be a vertex and $D_v$ its complete resolution diagram. We have a cobordism $V_u:\Sigma(D^\dagger)\longrightarrow \Sigma(D_u^\dagger)$ formed by the the addition of $|u|_1$-many 2-handles. If $D$ is oriented, there is a distinguished resolution $u_\ast\in\{0,1\}^n$, the {\emph{oriented resolution}}, and the associated resolution diagram $D_{\ast}$ inherits an orientation $o_\ast$ from $D$, and hence $\Sigma(D_{\ast}^\dagger)$ an associated spin structure $\mathfrak{s}_\ast$. From the proposition, the cobordism $V_{\ast}$ from $\Sigma(L^\dagger)$ to the branched cover of the oriented resolution is spin, bounding $\mathfrak{s}$ and $\mathfrak{s}_\ast$.

Now let $\omega$ be a two-fold marking for $L$, and suppose $L$ has $\ell$ components. The number (\ref{eq:gradings}) can be seen combinatorially as follows. The difference of two orientations $o_1$ and $o_2$ of $L$ can be viewed as a vector in $\{0,1\}^\ell$, by comparing orientations componentwise. In fact, the quantity
\begin{equation}
	\langle o_1 - o_2, \; \omega\rangle\; \equiv\; \sum_{i=1}^\ell (o_1^i -o^i_2)\cdot \omega(K_i) \mod 2, \label{eq:ordiff}
\end{equation}
where $o_1^i$ is the orientation of component $K_i\subset L$, is easily seen to depend only on the quasi-orientations determined by $o_1$ and $o_2$. This number is the same as (\ref{eq:gradings}) for the spin structures on $\Sigma(L^\dagger)$ determined by $\pm o_1$ and $\pm o_2$.\\
\vspace{.25cm}

\begin{figure}[t]
\centering
\includegraphics[scale=.60]{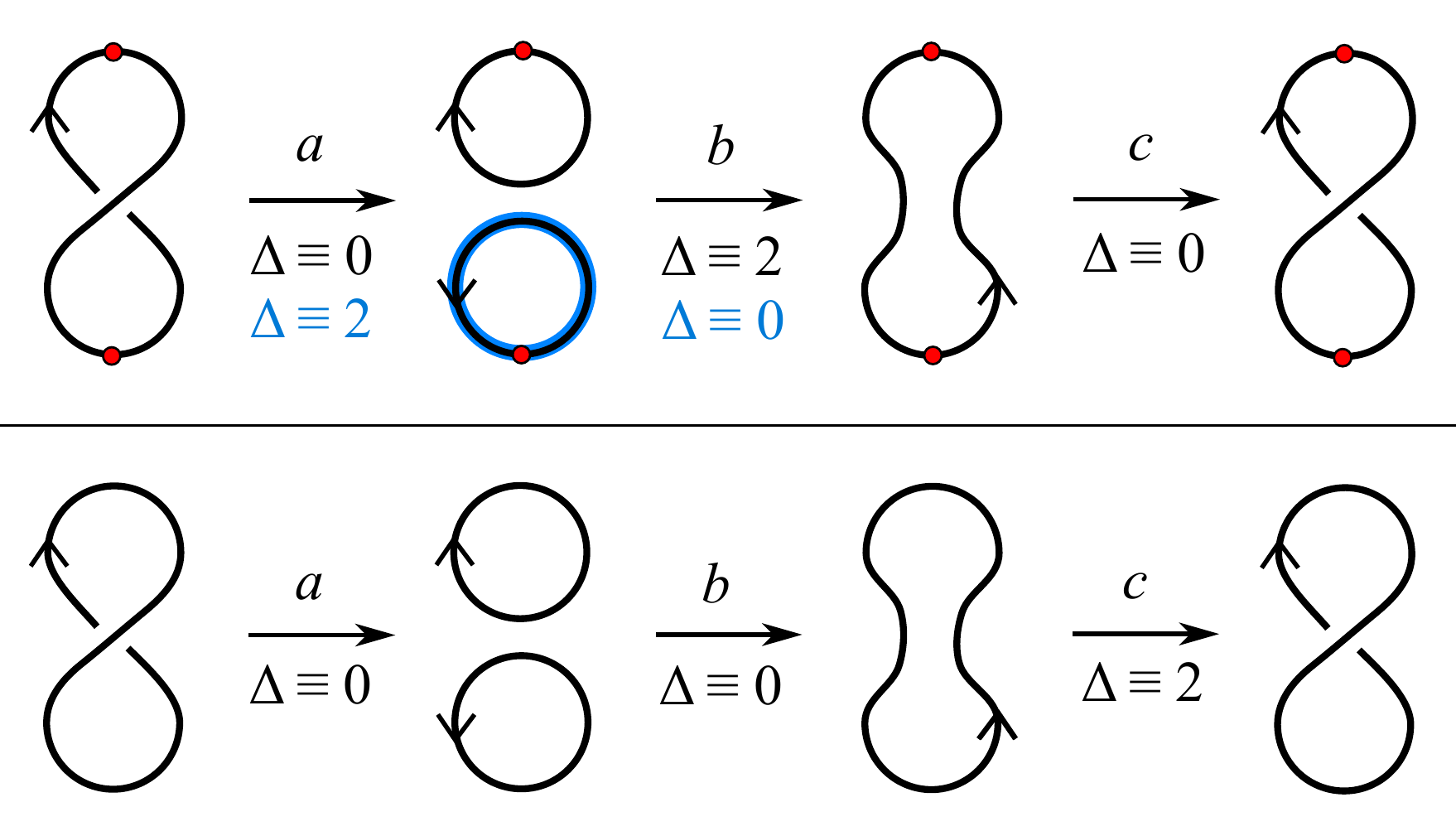}
\caption{In the top row, we fit the twisted unknot into an exact triangle in which the two-circle resolution has non-trivial marking data. On the bottom, the same exact triangle, but without dots. The mod 2 degrees of cobordism maps, and more specifically the $d(X)$ terms, are unchanged from the top row to the bottom row. We list in the figure the $\Delta \in 2\cdot \Z_4$ terms for each cobordism map. If we alter the quasi-orientation of the two-circle resolution in the top row, by changing the orientation of the bold (blue) circle, then the two surrounding $\Delta$ terms change by $2$ (mod 4).}\label{fig:unknottriangle}
\end{figure}

\begin{example}[The twisted unknot]\label{example:basic} {\emph{
\noindent Let $(D,\check{\omega})$ be an oriented two-fold marked diagram of the unkot with one twist and a dot on each arc, as in the top row of Figure \ref{fig:unknottriangle}. The $0$-resolution is the oriented resolution $u_\ast$, and as before we have the cobordism $V_0:\Sigma(D^\dagger)\longrightarrow\Sigma(D_0^\dagger)$. This map induces in framed instanton homology the map $a$ of Figure \ref{fig:unknottriangle}. We compute 
\[
	\text{deg}[\mathfrak{s},\mathfrak{s}_0](a) \equiv d(V_0) \equiv -\frac{3}{2}(1 + 0) +\frac{1}{2}(1-0) \equiv -1 \mod 4,
\]
where $S_0$ is a core for the attached 2-handle. The $\Delta$ term is $0$ (mod 4) by Prop. \ref{prop:deg1}. Note that the choice of $\mathfrak{s}$ is unique. Now, we denote by $\circlearrowleft$ the other quasi-orientation of $D_0$, obtained by reversing one of the orientations of the circles in $D_0$, as in Figure \ref{fig:unknottriangle}. We write $\mathfrak{s}_\circlearrowleft$\vspace{-.10cm} for the corresponding spin structure on $\Sigma(D^\dagger_0)$. Using (\ref{eq:ordiff}), we compute that when $\mathfrak{s}_i$ is replaced by $\mathfrak{s}_\circlearrowleft$, the degree (and more specifically, the $\Delta$ term) changes by $2$:
\[
	\text{deg}[\mathfrak{s},\mathfrak{s}_\circlearrowleft](a) \equiv -1 + 2 \equiv 1 \mod 4.
\]
The other $\Delta$ terms in the exact sequences are given in Figure \ref{fig:unknottriangle}, and follow from \cite{scaduto}. We note that in the exact triangle, for the triple composite cobordism, one always has $\Delta \equiv 2$ (mod 4) and $d \equiv 1$ (mod 4), giving the total degree of $-1$ (mod 4). Notice that an outcome of our discussion is the following fact. View $S^1\times S^2$ as the double branched cover of two unknots in the plane which are not concentric. Then the spin structure for $S^1\times S^2$ which spin bounds $D^2\times S^2$ is the one corresponding to the quasi-orientation in which the two circles have opposite planar orientations. (If the circles are concentric, the planar orientations should agree.)}}
\end{example}

\vspace{.25cm}

\begin{example}[The canonical quasi-orientation of a complete resolution]\label{example:canor} {\emph{
\noindent The previous example generalizes in the following way. We let $(D,\check{\omega})$ be a two-fold marked diagram. For any complete resolution $D_u$, we define the following:\\
}}
\end{example}

\begin{figure}[t]
\centering
\includegraphics[scale=.77]{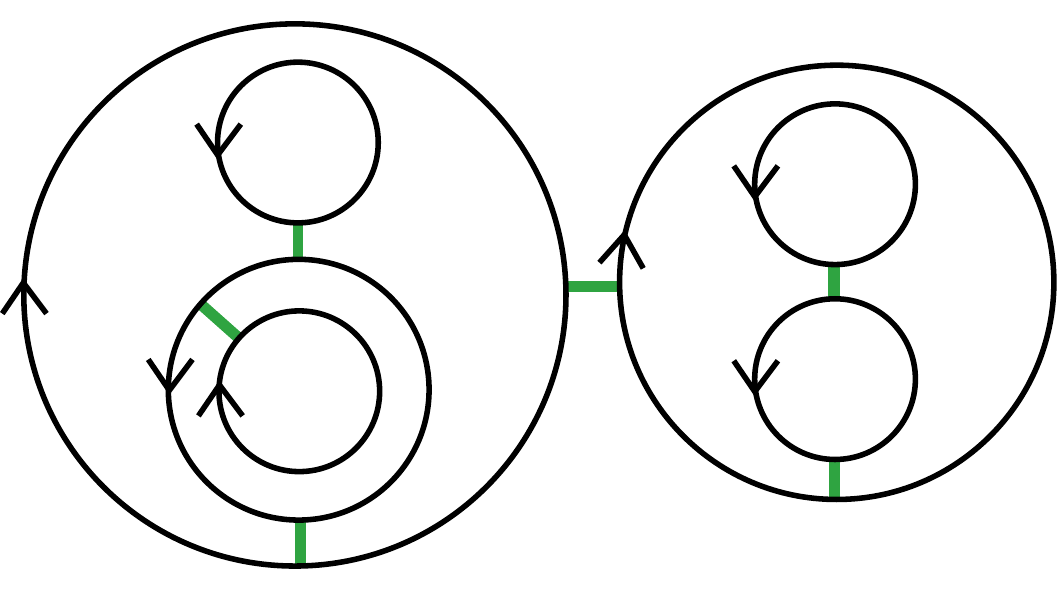}
\caption{An example of a complete resolution diagram with the canonical quasi-orientation. If two circles can be connected by an arc which does not intersect any other circles, then the orientations are chosen to agree with the right-hand side of Figure \ref{fig:orientedmoves}.}\label{fig:canonicalquasiorientation}
\end{figure}

\vspace{-.40cm}
\begin{defn}
The {\emph{canonical quasi-orientation of $D_u$}} is represented by one of two orientations that satisfy the following: if any two circles in $D_u$ are joined by an arc which does not interstect any other circles in the plane, then the orientations adhere to the local picture in the right of Figure \ref{fig:orientedmoves}. The canonical quasi-orientation will always be denoted by $\circlearrowleft$.\\
\end{defn}

\noindent An example of the canonical quasi-orientation is given in Figure \ref{fig:canonicalquasiorientation}. Now, suppose $(D_u,\check{\omega}_u)$ and $(D_w,\check{\omega}_w)$ are two complete resolutions of $(D,\check{\omega})$ where $u,w\in \{0,1\}^n$. We suppose $|u-w|_1=1$ and $u\leq w$, so that there is a single merge or split from $D_u$ to $D_w$. Let $V_{uw}:\Sigma(D_u^\dagger)\longrightarrow\Sigma(D_w^\dagger)$ be the relevant cobordism with bundle data $S_{uw}$. Then, regardless of whether any of the relevant circles merging or splitting support dots, we have
\[
	\Delta[\mathfrak{s}_{\circlearrowleft},\mathfrak{s}_{\circlearrowleft}](I^\#(V_{uw},S_{uw})) \equiv 0 \mod 4.
\]
Thus, when we use the canonical quasi-orientation on all resolutions, the $\Delta\in 2\cdot\Z_4$ terms of the merge and split cobordism maps are all zero.\\

\section{The mod 4 graded spectral sequence.}\label{sec:ss}

We now have the tools to establish the spectral sequence with mod 4 gradings. Suppose we have a two-fold marked diagram $(D,\check{\omega})$ with orientation $o$, yielding a distinguished oriented resolution diagram $(D_\ast,\check{\omega}_\ast)$ oriented by $o_\ast$. Then define $\theta = \theta(D,\check{\omega},o)\in \Z_2$ to be
\[
	\theta \;\equiv\; \langle o_\ast - \circlearrowleft, \check{\omega}_{u_\ast}\rangle \mod 2,
\]
in the sense of (\ref{eq:ordiff}), where, as before, $\circlearrowleft$ is the canonical quasi-orientation.\\

\vspace{.25cm}

\begin{theorem}\label{thm:ss}
	Let $L$ be a link with two-fold marking data $\omega$. Let $\pm o$ be a quasi-orientation for $L$, and write $\mathfrak{s}$ for the associated spin structure on $\Sigma(L)$. Let $(D,\check{\omega})$ be a dotted surgery diagram for $(L,\omega)$ oriented by $\pm o$. Then there is a spectral sequence with $E^2$-page $\text{{\emph{Hd}}}(D,\check{\omega})$ mod 4 graded by\footnote{This contains as a special case the grading in Theorem 1.1 of \cite{scaduto}. The grading there should be corrected by the addition of $2\sigma$ (mod 4), due to an error in Lemma 8.5, which should read $\mathcal{P}(\mathbb{X}_{\infty 1})\equiv n_-$ (mod 2).}
\begin{equation}
	 \delta^\# := 2h - \delta + 2\theta   \mod 4\label{eq:grinproof}
\end{equation}
that converges to {\emph{$I^\#(\Sigma(L^\dagger),\omega)$}} with mod 4 grading $\text{{\emph{gr}}}[\mathfrak{s}]$.\\
\end{theorem}
\vspace{.10cm}

\begin{proof} From Proposition \ref{prop:ss}, we need only compute the mod 4 gradings. Let $x=a_1\wedge \cdots \wedge a_i \in C_u$ be an element in the $E^1$-page. We assume $x\neq 0$, in which case $\check{\omega}_u$ must have associated $SO(3)$-bundle trivial. Thus all choices of spin structure $\mathfrak{s}_u$ on $\Sigma(D_u)$ yield the same mod 4 grading in $C_u \cong I^\#(\Sigma(D^\dagger_u))$. Following \cite[\S 6.4]{scaduto}, the $E^1$-page mod 4 grading of $x$ is given by
\begin{equation}
	2(k-1) + i - |u|_1 - \text{deg}[\mathfrak{s},\mathfrak{s}_u]\left( m_{u}\right),\label{eq:e1grad}
\end{equation}
where $k$ is the number of circles in the resolution $D_u$, and $m_{u}=I^\#(X_{ u},S_{u})$, where $(X_{u},S_{u})$ is a cobordism from $(\Sigma(D^\dagger),\check{\omega})$ to $(\Sigma(D^\dagger_u),\check{\omega}_u)$. The cobordism $X_{u}$ is formed as follows. Let $\vec{0}\in \{0,1\}^n$ be the resolution of norm $0$. Form the cobordism $V_{\vec{0}}$ as above from $\Sigma(L^\dagger)$ to $\Sigma(D_{\vec{0}}^\dagger)$ by attaching $n$ 2-handles to $\Sigma(L^\dagger)\times [0,1]$. Then we attach $|u|_1$ more 2-handles to obtain $X_u$, one for each 1-smoothing of $u$. Plainly $\chi(X_{u}) = |u|_1 + n$, and by \cite[Lemma 9.4]{bloom} (and additivity of signatures) we have $\sigma(X_{u}) = -n_- -\sigma$. This computes
\[
	d(X_{u}) = -\frac{3}{2}(n_+ - \sigma + |u|_1) + \frac{1}{2}( (k-1)-\nu).
\]
From this we then compute (\ref{eq:e1grad}) to be mod 4 congruent to
\[
	\Delta_{u} - \delta +2|u|_1,
\]
where $\delta$ is the grading defined in (\ref{eq:deltagrad}), and $\Delta_{ u}$ is the $\Delta$ term for the map $m_u$, i.e.
\[
    \Delta_u = \text{deg}[\mathfrak{s},\mathfrak{s}_u]\left( m_{u}\right)-\text{deg}(X_{u})\;\in\; 2\cdot \Z_4.
\]
Recalling $h=|u|_1-n_-$ and $|u_\ast|_1 = n_-$, the following will establish (\ref{eq:grinproof}) and complete the proof:
\[
	\Delta_{u} \equiv 2|u_\ast|_1+ 2\theta \mod 4.
\]
Up until now the spin structure choices $\mathfrak{s}_u$ did not matter, as we were working exclusively at a trivially marked diagram $D_u$. We will utilize the additivity of degrees, and in doing so must pass through resolutions with non-trivial two-fold marking data. We thus choose for every vertex $u\in \{0,1\}^n$ the spin structure $\mathfrak{s}_u = \mathfrak{s}_\circlearrowleft$ corresponding to the canonical quasi-orientation $\circlearrowleft$. Now consider the case in which $u$ is the oriented resolution vertex $u_\ast$. Then $(X_u,S_u)$ is none other than $(V_\ast, S_\ast)$ blown up $|u_\ast|_1$ many times by $(\overline{\mathbb{C}\mathbb{P}}^2, F)$, cf. \cite[Fig. 3.2]{scaduto}. We then compute
\[
	\Delta_{u_\ast} \equiv 2 |u_\ast|_1  + 2 \theta \mod 4,
\]
where the first term comes from Proposition \ref{prop:blowup}, and the second term is computed from Proposition \ref{prop:deg}, using that $V_\ast$ is a spin cobordism bounding $\mathfrak{s}$ and $\mathfrak{s}_\ast$, thereafter utilizing (\ref{eq:ordiff}) to compare $\mathfrak{s}_\circlearrowleft$ and $\mathfrak{s}_\ast$ against $\check{\omega}_\ast$. 
From Example \ref{example:canor}, the $\Delta$ term is $0$ (mod 4) for a merge or split when both diagrams are canonically quasi-oriented. Using additivity, a sequence of such merges and splits shows that $\Delta_u \equiv \Delta_{u_\ast}$ (mod 4), completing the proof.
\end{proof}
\vspace{.50cm}

\begin{figure}[t]
\centering
\includegraphics[scale=.23]{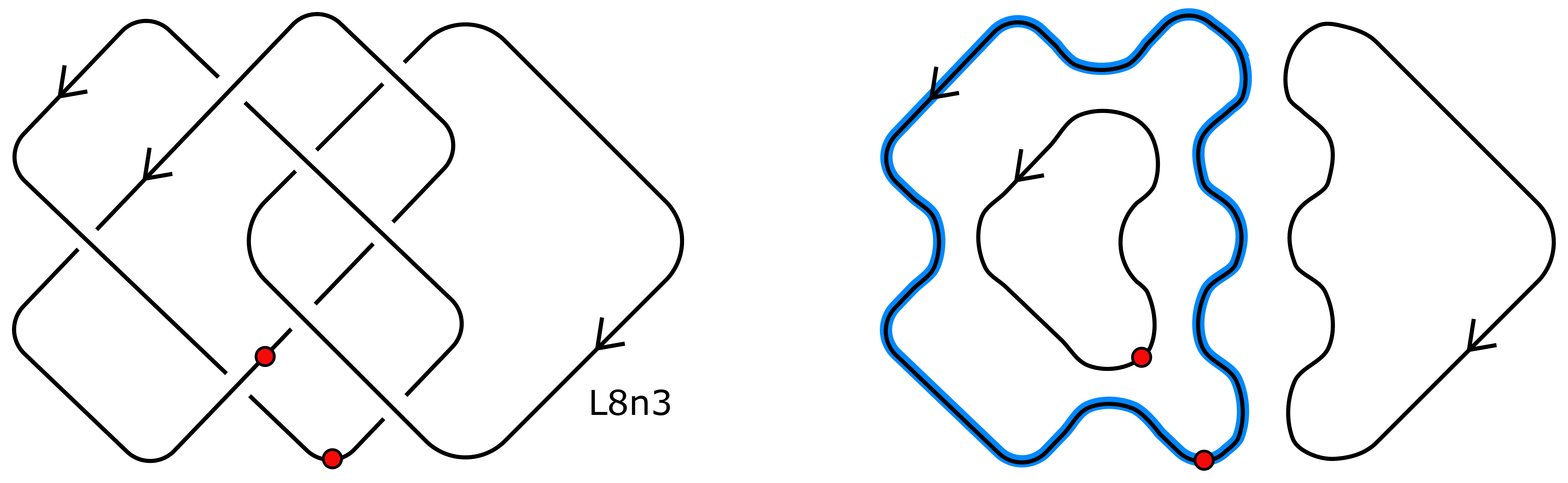}
\caption{We compute the term $\theta$ in the mod 4 grading for the two-fold marked diagram $(D,\check{\omega})$ of L8n3 depicted on the left. First we take the oriented resolution $(D_\ast,\check{\omega}_\ast)$, depicted on the right. If the circle in bold (blue) has its orientation reversed, we obtain the canonical quasi-orientation. There is a dot on this circle, so we compute $\theta \equiv 1$ (mod 2).}\label{fig:L8n3}
\end{figure}

\noindent When considering the isomorphism class of the framed instanton homology, the important part of the grading $\delta^\#$ is $2h-\delta$. This determines the relative $\Z_4$-grading, and the other term $2\theta$ only contributes an overall shift. In fact, when $\omega$ is non-trivial, all of our computations indicate that $\text{Hd}(D,\check{\omega})$ is 2-periodic with respect to the $\Z_4$-grading $\delta^\#$. The influence of the homological term $2h$ is understood when we use non-trivial bundles, as illustrated through the following simple examples.\\

\vspace{.25cm}

\noindent {\textbf{The Hopf link L2a1.}} For any two-fold marking $\check{\omega}$ of the alternating diagram $D$ for the Hopf link, $\text{Hd}(D,\check{\omega})$ computes $\text{Kh}(L,\omega)$ as well as $I^\#(\R\mathbb{P}^3,\omega)$. Since $L$ is alternating, $\text{Kh}(L,\omega)$ is rank 2 and in $\delta$-grading 0. When the marking $\check{\omega}$ is trivial, and we use the grading $2h-\delta$ (mod 4), $\text{Hd}(D,\check{\omega})$ is rank 2 and supported in grading 0 (mod 4), which computes $I^\#(\R\mathbb{P}^3)$. If the marking $\check{\omega}$ represents a non-trivial bundle, then $2h$ influences the result more: we compute $I^\#(\R\mathbb{P}^3,\omega)$ to be rank 1 in grading 0 (mod 4) and rank 1 in grading 2 (mod 4). The 2-periodicity allows us to forget $\theta$, if we only care about rank inequalities. However, we note that if $\check{\omega}$ is non-trivial but represents the trivial bundle (2 dots on one component), the $2h-\delta$ (mod 4) graded $\text{Hd}(D,\check{\omega})$ is rank 2 in grading 2 (mod 4). Here the term $2\theta\equiv 2$ (mod 4), and shifts the answer to what it should be.\\

\vspace{.25cm}

\noindent {\textbf{The 2-bridge link L4a1.}} Similarly, the alternating diagram for L4a1 with trivial marking data computes $\text{Kh}(L,\omega)$ to be rank 4 and supported in $\delta$-grading 0. With the grading $2h-\delta$ (mod 4), the dotted diagram homology computes $I^\#(L(4,1))$ to be of rank 4, supported in grading 2 (mod 4). When $\check{\omega}$ represents a non-trivial bundle, $2h$ alters the gradings to compute that $I^\#(L(4,1),\omega)$ is of rank 2 in gradings 0 (mod 4) and 2 (mod 4). This latter case can be seen in Figure \ref{fig:L4a1}.\\

\vspace{.25cm}

\noindent {\textbf{The non-alternating link L6n1.}} The first two examples above are of course both implied by Theorem \ref{thm:tqa}, as is the two-fold marked L6n1 of Figure \ref{fig:L6n1comp}, which is TQA. Write $(L,\omega)$ for this example. The dotted diagram homology computes $I^\#(\Sigma(L),\omega)$ to be of rank 4. The $\delta$-grading is zero for all generators. The bold (green) homology in Figure \ref{fig:L6n1comp} shows that the generators pair off into different columns ($h$-gradings). Using the $2h-\delta$ (mod 4) grading, this gives that $I^\#(\Sigma(L),\omega)$ is of rank 2 in both gradings 0 (mod 4) and 2 (mod 4).\\

\vspace{.25cm}

\noindent {\textbf{The non-alternating link L8n3.}} This example does not follow from Theorem \ref{thm:tqa}. We have $\det(\text{L8n3})=4$. With our conventions, the signature of the oriented L8n3 represented by the two-fold marked diagram $(D,\check{\omega})$ in Figure \ref{fig:L8n3} is $-6$. With $\F = \Z_2$ as usual, the $2h-\delta$ (mod 4) graded dotted diagram homology $\text{Hd}(D,\check{\omega})$ is computed to be
\begin{equation}
    \F_{(0)}^3 \oplus \F_{(1)} \oplus \F_{(2)}^3 \oplus \F_{(3)}.\label{eq:agroup}
\end{equation}
We then compute in Figure \ref{fig:L8n3} that $\theta\equiv 1$ (mod 2). Thus the gradings of generators in the expression (\ref{eq:agroup}) are shifted by 2 (mod 4). Once again, the shift by $2\theta$ has no effect on the $\Z_4$-graded isomorphism class of the group.\\

\vspace{.0cm}

\section{Framed instanton homology for TQA marked links}\label{sec:tqa}

In this section we prove Theorem \ref{thm:tqa}. We begin with two facts about framed instanton homology. First, the euler characteristic is given by $\chi(I^\#(Y,\omega)) = |H_1(Y;\Z)|$, see \cite{scaduto}. Second, $I^\#(Y,\omega)$ vanishes whenever there is a 2-sphere in $Y$ that transversely intersects $\omega$ in an odd number of points, i.e. the $SO(3)$-bundle associated to $\omega$ is non-trivial on an embedded 2-sphere. This latter fact holds on the chain level, and comes down to the fact that a 2-sphere supports no flat connections on a non-trivial $SO(3)$-bundle. For double branched covers, we thus have $\chi(I^\#(\Sigma(L),\omega))=\det(L)$, and that $I^\#(\Sigma(L),\omega)$ vanishes whenever $(L,\omega)$ splits into two odd-marked links.\\

\begin{proof}[Proof of Theorem \ref{thm:tqa}]
Let $(L,\omega)$ be a TQA two-fold marked link. Let the resolution two-fold marked links $(L_0,\omega_0)$ and $(L_1,\omega_1)$ be as in Definition \ref{defn:tqa}. By symmetry, the loop $\mu$ in the exact triangle (\ref{eq:ses}) can be placed in any one of the instanton groups. We ensure that $\mu$ is always mod 2 nullhomologous as follows: if in the relevant crossing of $L$ the two strands belong to one component, add $\mu$ to $(L,\omega)$; otherwise add $\mu$ to either $(L_0,\omega_0)$ or $(L_1,\omega_1)$. Thus up to isomorphism the exact triangle (\ref{eq:ses}) can be presented in the following form, analogous to (\ref{eq:skeintriangle}):\\
\begin{equation*}
\begin{tikzcd}
                \cdots \;\;I^\#(\Sigma(L^\dagger),\omega) \arrow{r} &  I^\#(\Sigma(L^\dagger_0),\omega_0) \arrow{r}     &  I^\#(\Sigma(L^\dagger_1),\omega_1)              \arrow{r}                 &  I^\#(\Sigma(L^\dagger),\omega)\;\; \cdots 
\end{tikzcd}
\end{equation*}
\vspace{.15cm}

\noindent Now, if no three of these two-fold marked links are split into odd marked links, we can induct on determinants, using the above-mentioned facts, to establish that $I^\#(\Sigma(L),\omega)$ is $\Z_2$-graded isomorphic to a free abelian group of rank $\det(L)$. The induction can still be carried through in general, with a few modifications, as follows. 

Suppose first that $\det(L)=0$. Then $L_0$ and $L_1$ also have determinant zero. By resolving $(L_0,\omega_0)$ and $(L_1,\omega_1)$ at TQA crossings, and continuing, we obtain a tree. We can choose the resolutions so that the tree is finite, and the leaves of this tree are two-fold marked links that are split into odd-marked links. All such marked links have zero instanton homology, and using the exact triangle, we deduce the same for $L$. 

Now assume that $\det(L)> 0$. Suppose first that one of $L_0$ or $L_1$ has determinant zero. Let $(L_+,\omega_+)$ be the two-fold marked link amongst $(L_0,\omega_0)$ and $(L_1,\omega_1)$ that has positive determinant. We choose a TQA resolution of $(L_+,\omega_+)$, and if again there is a determinant zero link involved, we again choose the positive determinant marked link, and so forth. We can assume this process terminates, in the sense that eventually we come to a resolution in which either (i) the two resolved links $L'_0$ and $L'_1$ both have positive determinant, or (ii) one of $L'_0$ or $L'_1$ is split odd-marked and the other is an unknot. Case (ii) occurs only if $\det(L)=1$. If we can argue that the instanton groups for $(L,\omega)$ and $(L_+,\omega_+)$ are $\Z_4$-graded isomorphic, then upon passing through the aforementioned sequence of resolutions, we can reduce to proving the induction step for the cases (i) and (ii). The result for case (ii) would follow immediately. 

It is clear upon applying the exact triangle that $I^\#(\Sigma(L^\dagger),\omega)$ and $I^\#(\Sigma(L_+^\dagger),\omega_+)$ are $\Z_2$-graded isomorphic. However, we want a $\Z_4$-graded isomorphism. For this purpose we first establish:\\

\vspace{.25cm}

\noindent \textbf{Claim:} {\emph{In the above situation, the marking $\omega$ is trivial if and only if $\omega_+$ is trivial}}.\\

\begin{proof} If the two strands at the resolution crossing of $L$ belong to the same component, then this is easy to see. If the two strands belong to different components $K$ and $K'$ of $L$, we argue as follows. Write $(L_-,\omega_-)$ for the resolution other than $(L_+,\omega_+)$, for which $\det(L_-)=0$. If $\omega$ is trivial, then clearly so are $\omega_+$ and $\omega_-$. Now, for contradiction, suppose that $\omega$ is non-trivial but $\omega_+$ is trivial. Then there must be no dots on components of $L$ other than $K$ and $K'$, and one dot on each of $K$ and $K'$. From this it is clear that both $\omega_+$ and $\omega_-$ are trivial. As the theorem is already proved for determinant zero TQA marked links, we know that $\kh(L_-,\omega_-)=\overline{\kh}(L_-;\F)=0$. But reduced Khovanov homology has graded euler characteristic the Jones polynomial $J_L(t)$, and $J_L(1)=2^{|L|-1}$ implies $\overline{\kh}(L_-;\F)\neq 0$. This provides a contradiction, proving the claim.
\end{proof}

\vspace{.35cm}

\noindent Now we return to establishing the $\Z_4$-graded isomorphism between $I^\#(\Sigma(L^\dagger),\omega)$ and $I^\#(\Sigma(L_+^\dagger),\omega_+)$. If $\omega$ is non-trivial, the $\Z_2$-graded isomorphism suffices, because we claim, in the end, that its isomorphism type is 2-periodic with respect to the $\Z_4$-grading. When $\omega$ is trivial, the relevant map induced by a cobordism $(V_+,S_+)$ between the two groups has as its model the map $c$ in the top row of Figure \ref{fig:unknottriangle}. Thus $\Delta\equiv 0$ (mod 4). On the other hand $d(V_+)\equiv 0$, cf. \cite[\S 42.3]{thebook}. Thus the degree of the isomorphism between the two groups is $0$ (mod 4), whence they are $\Z_4$-graded isomorphic.

Now we can assume that we are in case (i) above, i.e., that all three links $L,L_0$ and $L_1$ have positive determinant. We have a few cases to consider, depending on the nature of the resolutions $(L_0,\omega_0)$ and $(L_1,\omega_1)$. To make sense of absolute mod 4 gradings, quasi-orient $L$ by  $\pm o$, and let $\mathfrak{s}$ denote the associated spin structure on $\Sigma(L^\dagger)$. Then $I^\#(\Sigma(L^\dagger),\omega)$ has the absolute $\Z_4$-grading $\text{gr}[\mathfrak{s}]$. We proceed to check the induction step in the following cases: \\

\vspace{.15cm}

\noindent {\emph{(1) Both $\omega_0$ and $\omega_1$ are non-trivial, and so too is $\omega$.}} By the induction hypothesis on determinants, we have an exact sequence of $\Z_4$-graded abelian groups:\\
\begin{equation*}
\begin{tikzcd}
               0 \arrow{r} & \Z_{(0)}^{\frac{1}{2}\det(L_1)}\oplus\Z_{(2)}^{\frac{1}{2}\det(L_1)} \arrow{r}     &  I^\#(\Sigma(L^\dagger),\omega)     \arrow{r}      &  \Z_{(0)}^{\frac{1}{2}\det(L_0)}\oplus \Z_{(2)}^{\frac{1}{2}\det(L_0)} \arrow{r} & 0
\end{tikzcd}
\end{equation*}
\vspace{.15cm}

\noindent The degrees of the two non-trivial maps are even. Regardless of their actual mod 4 degrees, this implies the induction step for $(L,\omega)$ in this case.\\

\vspace{.15cm}

\noindent {\emph{(2) Both $\omega_0$ and $\omega_1$ are trivial, but $\omega$ is non-trivial.}} In this case, the two strands of the crossing being resolved must belong to distinct link components of $L$, say $K_1$ and $K_2$, and $\omega$ must be the two-fold marking data which corresponds to a dot on each of these components (and otherwise no dots). Choose the orientation $o$ of $L$ so that $L_1$ is the oriented resolution. Note in our situation that $|L_0|=|L_1|=|L|-1$. Then, setting $\varepsilon_j  = (-1)^{j/2}2^{|L|-2}$, we have\\
\begin{equation}
\begin{tikzcd}
               0 \arrow{r} & \displaystyle{\bigoplus_{j=0,2}}\Z_{(j)}^{\frac{1}{2}\left(\det(L_1)+\varepsilon_j \right)}\arrow{r}{a}    &  I^\#(\Sigma(L^\dagger),\omega)     \arrow{r}{b}      & \displaystyle{\bigoplus_{j=0,2}}\Z_{(j)}^{\frac{1}{2}\left(\det(L_0)+\varepsilon_j \right)} \arrow{r} & 0\label{eq:trivses}
\end{tikzcd}
\end{equation}
\vspace{.15cm}

\noindent by our induction hypothesis on determinants. By our choice of $o$, we have $\text{deg}[\mathfrak{s}](a)\equiv 0$. Let $o'$ be obtained from $o$ by reversing the orientation of $K_1$, and write $\mathfrak{s}'$ for the associated spin structure on $\Sigma(L)$. Then $L_0$ is the oriented resolution for $o'$, and thus $\text{deg}[\mathfrak{s'}](b)\equiv 0$. Since $\omega(K_1)\equiv\omega(K_2)\equiv 1$ and $o$ and $o'$ differ on $K_1$, (\ref{eq:ordiff}) computes $\text{deg}[\mathfrak{s}](b)\equiv 2$. Now, using $\varepsilon_0+\varepsilon_2 = 0$ and the assumption that $\det(L_0)+\det(L_1)=\det(L)$, the induction step for this case follows.\\

\vspace{.15cm}

\noindent {\emph{(3) Only one of $\omega_0$ and $\omega_1$ is non-trivial, and $\omega$ is trivial.}} In this case, the two strands of the crossing must belong to the same component of $L$, say $K_1$. Suppose $\omega_0$ is non-trivial; the other case is similar. Set $\varepsilon_j = (-1)^{j/2}2^{|L_0|-1}$. We then have by our induction hypothesis:\\
\begin{equation*}
\begin{tikzcd}
               0 \arrow{r} & \displaystyle{\bigoplus_{j=0,2}}\Z_{(j)}^{\frac{1}{2}\left(\det(L_1)+\varepsilon_j \right)}\arrow{r}{a}   &  I^\#(\Sigma(L^\dagger),\omega)     \arrow{r}{b}      & \Z_{(0)}^{\frac{1}{2}\det(L_0)}\oplus \Z_{(2)}^{\frac{1}{2}\det(L_0)}  \arrow{r} & 0
\end{tikzcd}
\end{equation*}
\vspace{.15cm}

\noindent We compute $\text{deg}[\mathfrak{s}](a)\equiv 0$ (which is independent of $\pm o$ and $\mathfrak{s}$, as the bundles involved are trivial). Knowing that the degree of $b$ is even, the additivity of determinants, and the observation that in this situation $|L|=|L_0|$, the induction step follows.\\

\vspace{.15cm}

\noindent {\emph{(4) Both of $\omega_0$ and $\omega_1$ are trivial, and so too is $\omega$.}} In this case we again have the exact sequence (\ref{eq:trivses}), except that the middle group has $\omega$ trivial. We know the degrees of the two maps $a$ and $b$ are even. First suppose that the two strands crossing being resolved belong to distinct components of $L$. Then the reasoning in case (2) implies that the degrees of $a$ and $b$ are both 0 (mod 4). The additivity of the determinants leaves us to verify that
\[
	\varepsilon_j(L) = \varepsilon_j(L_0) + \varepsilon_j(L_1), \qquad j\in \{0,2\}
\]
where $\varepsilon_j(L) = (-1)^{j/2}2^{|L|-1}$. Under our assumptions we have $|L|-1=|L_1|=|L_0|$, so this is easily verified. Now suppose instead that the two strands of the crossing belong to the same component of $L$. Orient $L$. Suppose that $L_0$ is the oriented resolution; the other case is similar. Then we have $|L| = |L_0| - 1= |L_1|$. Also, the map $a$ is degree 0 (mod 4). We argue that the degree of $b$ is 2 (mod 4). Let $c$ be the third map in the exact triangle besides $a$ and $b$. This map is the zero map, but its degree is still defined. Then by \cite[\S 4.6]{scaduto} we have
\[
    \text{deg}(a) + \text{deg}(b) + \text{deg}(c) \equiv -1 \text{\;(mod 4)}.
\]
It suffices to show that $\text{deg}(c)\equiv 1$ (mod 4). By changing the orientation of $L$ as in case (2), using Prop. \ref{prop:deg1} we see that the map $c$ actually has degree given by $d(V_{10})$ where $V_{10}:\Sigma(L^\dagger_1)\to \Sigma(L^\dagger_0)$ is the relevant cobordism. The rest follows from the mod 2 gradings, which we know, and \cite[\S 42.3]{thebook}. In particular, the signature of $V_{10}$ is +1, and we directly compute $d(V_{10})\equiv 1$ (mod 4). Finally, with the gradings of $a$ and $b$ understood, the additivity of determinants leaves us to show
\[
	\varepsilon_j(L) = \varepsilon_j(L_0) + \varepsilon_{j+2}(L_1), \qquad j\in \{0,2\}.
\]
This is easily verified, using $|L| = |L_0| - 1= |L_1|$.\\

\vspace{.20cm}

\noindent It is straightforward to verify that one cannot have only two of $\omega,\omega_0,\omega_1$ non-trivial. Thus the above four cases exhaust all possibilities, and the proof is complete.
\end{proof}

\vspace{.50cm}

\section{TQA Montesinos knots}\label{sec:more}

Here we present a list of TQA Montesinos knots, many of which are not QA. We adapt the work of Champarnerkar and Ording \cite[Prop. 5.1]{champanerkarording} from the QA case, and we follow their conventions and notation. We write $L=M(e;t_1,\dots,t_p)$ for a Montesinos link, where $0\neq t_i=\alpha_i/\beta_i\in \Q$, $\alpha_i,\beta_i\in \Z$ are relatively prime, $\beta_i >0$, and $e\in \Z$.  Let $\{t\}=t-\lfloor t\rfloor$ denote the fractional part of $t$. Define
\[\
    \varepsilon = e+\sum \left\lfloor\frac{1}{t_i}\right\rfloor, \;\;\; \hat{t}=\left\{\frac{1}{t}\right\}^{-1},\;\;\; t^f=\frac{\alpha}{\beta-\alpha} \; \text{ if } t>0,\;\;\; t^f=\frac{\alpha}{\beta+\alpha},\; \text{ if } t<0.
\]

\begin{prop}\label{prop:monte}
Let $L=M(e;t_1,\dots,t_p)$ be a Montesinos knot.  Then $L$ is TQA if any one of the three following conditions holds:
\begin{itemize}
\item $\varepsilon>-1$\;\; or \;\;$\varepsilon<1-p$.
\item $\varepsilon=-1$\; and\; $|\hat{t_i}^f| \geq \hat{t_j}$ \; for some $i\neq j$.  
\item $\varepsilon=1-p$\; and\; $|\hat{t_i}^f| \leq \hat{t_j}$ \; for some $i\neq j$.\\  
\end{itemize}
\end{prop}

\vspace{.25cm}

\noindent Qazaqzeh, Chbili, and Qublan \cite{qazaqzehchbiliqublan} show that a Montesinos link is QA if one of these conditions holds, with {\emph{strict inequality}} in the two latter conditions.  Further, they conjecture that all QA Montesinos links satisfy one of these conditions (with strict inequality).\\

\begin{proof}[Proof of Prop. \ref{prop:monte}]
The proof is based on the following proposition, which is an adaptation of Proposition 5.1 from \cite{champanerkarording} for QA links to our setting of TQA two-fold marked links. (Again, the difference is that our inequality is not strict.)\\

\begin{prop}\label{prop:premont1}
Let $s,r_i$ be positive rational numbers for $i=1,\dots,n$ such that $s\geq \min\{ r_i\}$.  Then there exists a choice of two-fold marking data $\omega$ on the Montesinos link $L=M(0;r_1,\dots,r_n,-s)$ so that $(L,\omega)$ is TQA.  In particular, if $L$ is a knot, then $L$ is TQA.\\
\end{prop}

\noindent The proof of this proposition is entirely analogous to that of \cite{champanerkarording}.  In the course of the proof, one needs that Theorem 2.1 of \cite{champanerkarkofman} also applies to $\mathcal{Q}_2$.  We list the appropriate generalization:\\

\begin{prop}[cf. \cite{champanerkarkofman} Theorem 2.1]\label{prop:champkof}
Let $(D,\check{\omega})$ be a two-fold marked diagram with a distinguished crossing as in Definition \ref{defn:tqa}, which we call $c$. Let $D'$ be $D$ with $c$ replaced by an alternating rational tangle that extends $c$. Then $(D',\check{\omega}')$ with the naturally induced marking data $\check{\omega}'$ represents a TQA two-fold marked link.\\
\end{prop}

\noindent Here we say that a rational tangle $[a_r,\dots,a_1]$ ``extends a crossing" if $\varepsilon(c)a_i\geq 1$, where $\varepsilon(c)$ denotes the sign of the crossing $c$. The proof of this proposition is unchanged from \cite{champanerkarkofman}.

Given Proposition \ref{prop:champkof} we prove Proposition \ref{prop:premont1}.  First, consider the base case $M(0;r_1,-s)$.  The case $s>r_1$ is dealt with by Champanerkar and Ording, and is QA. If $s=r_1$, then this is an unlink, and so we put on the nontrivial marking data, in which case the proposition is clear.

Now suppose that $(L_\infty,\omega_\infty)$, where $L_\infty=M(0;r_1,\dots,r_n,-s)$, is a TQA two-fold marked link which satisfies the condition of Prop. 6. We will show that $(L,\omega)$, where $L=M(0;r_1,\dots,r_n,1,-s)$, for some marking data $\omega$, is also TQA.  Indeed, upon using the standard Montesinos diagrams and resolving the $1$-crossing of $L$ (but not yet specifying the marking data), we obtain $L_\infty$ and 
\[
    L_0=1 \ast \overline{r}_1 \# \dots 1\ast \overline{r}_n \# 1 \ast s,
\]
following the notation of \cite{champanerkarording}.  The latter link $L_0$ is QA (so is TQA for any choice of marking data).  The condition $s\geq \min \{r_i\}$ guarantees additivity of determinants (cf. \cite[Prop. 5.1]{champanerkarording}):
\[
\mathrm{det}(L)=\mathrm{det}(L_0) +\mathrm{det}(L_\infty).
\]
By the induction hypothesis, we need only check that there exists a choice of two-fold marking $\omega$ that induces $\omega_\infty$ upon resolving the crossing.  This may be done by giving a dotted diagram whose marking data is $\omega_\infty$, and then leaving the appropriate dots on the diagram for $L$. Finally, by Prop. \ref{prop:champkof}, $(M(0;r_1,\dots,r_n,r_{n+1},-s),\omega')$ is TQA for any positive rational $r_{n+1}$, where $\omega'$ is induced from $L$.  Thus, we have shown that if all links of the form $M(0;r_1,\dots,r_n,-s)$ with $s\geq \min\{ r_i\}$ have some choice of marking data for which they are TQA, then all links of the form $M(0;r_1,\dots,r_n,r_{n+1},-s)$ have some choice of marking data for which they are TQA. Noting that knots have only trivial marking data, the proof of Prop. \ref{prop:premont1} is complete.

To prove Proposition \ref{prop:monte}, we note that the Montesinos link $L=M(e;t_1,\dots,t_p)$ is equivalent to the link $M(\varepsilon;\hat{t}_1,\dots,\hat{t}_p)$.  We may assume $\varepsilon=-1$, since $\varepsilon>-1$ links are in fact already QA by \cite{champanerkarording}, and $\varepsilon=1-p$ is the reflected case.  By Lemma 3.3 of \cite{champanerkarording}, we have
\[
M(\varepsilon;\hat{t}_1,\dots,\hat{t}_p)=M(\varepsilon+1;\hat{t}_1,\dots,\hat{t}_{i-1},\hat{t}^f_i,\hat{t}_{i+1},\dots, \hat{t}^f_p).
\]
If $L$ is a knot, the latter side is TQA by Prop. \ref{prop:premont1}.  This completes the proof.
\end{proof}
\vspace{.35cm}

\noindent We remark that when there are multiple components of the Montesinos link, the two-fold marking data in Proposition \ref{prop:premont1} cannot in general be trivial. For example, $M(0;2,2,-2)$ is the link L6n1 of Figure \ref{fig:L6n1isinQ2}. We saw that this link is TQA only when marked non-trivially.

\vspace{.60cm}

\section{Computer calculations} \label{sec:computer}

Using code written by the authors in {\texttt{Python}}, we computed the twisted Khovanov homology $\kh(L,\omega)$ for prime links $L$ with more than one component, up to 10 crossings, and all two-fold markings $\omega$. We list in Table \ref{table} all such $(L,\omega)$ for which $\kh(L,\omega)$ is not of rank $\det(L)$, i.e. {\emph{non-thin}}. In each case we verified Conjecture \ref{conj:1} by finding a dotted diagram $(D,\check{\omega})$ representing $(L,\omega)$ for which the spectral sequence from $\text{Hd}(D,\check{\omega})$ to $\kh(L,\omega)$ collapses. We write $\kh(L,\omega)_{\delta^\#}$ for the $\Z_4$-graded group $\kh(L,\omega)=\text{Hd}(D,\check{\omega})$ as in (\ref{eq:z4group}). The 4-tuples $(b_0,b_1,b_2,b_3)$ appearing in the table are the four betti numbers of $\kh(L,\omega)_{\delta_\#}$ indexed by $\Z_4 = \{0,1,2,3\}$:
\[
    b_i = \text{dim}_\F\kh(L,\omega)_{(i)}, \quad i\in \Z_4.
\]
In each case, by Theorem \ref{thm:ss}, the betti number $b_i$ gives an upper bound for the rank of the framed instanton homology with $\F=\Z_2$-coefficients:
\[
 b_i \geq \text{dim}_\F I^\#(\Sigma(L^\dagger),\omega)_{(i)}.
\]
For the convenience of the reader we have included for each link its number of components, its determinant, as well as its signature. We use the orientation of $L$ which is the first orientation listed in {\emph{Linkinfo}} \cite{linkinfo}. Our convention is that the signature of the right-handed trefoil is $+2$.

\begin{figure}[t]
\centering
\includegraphics[scale=.13]{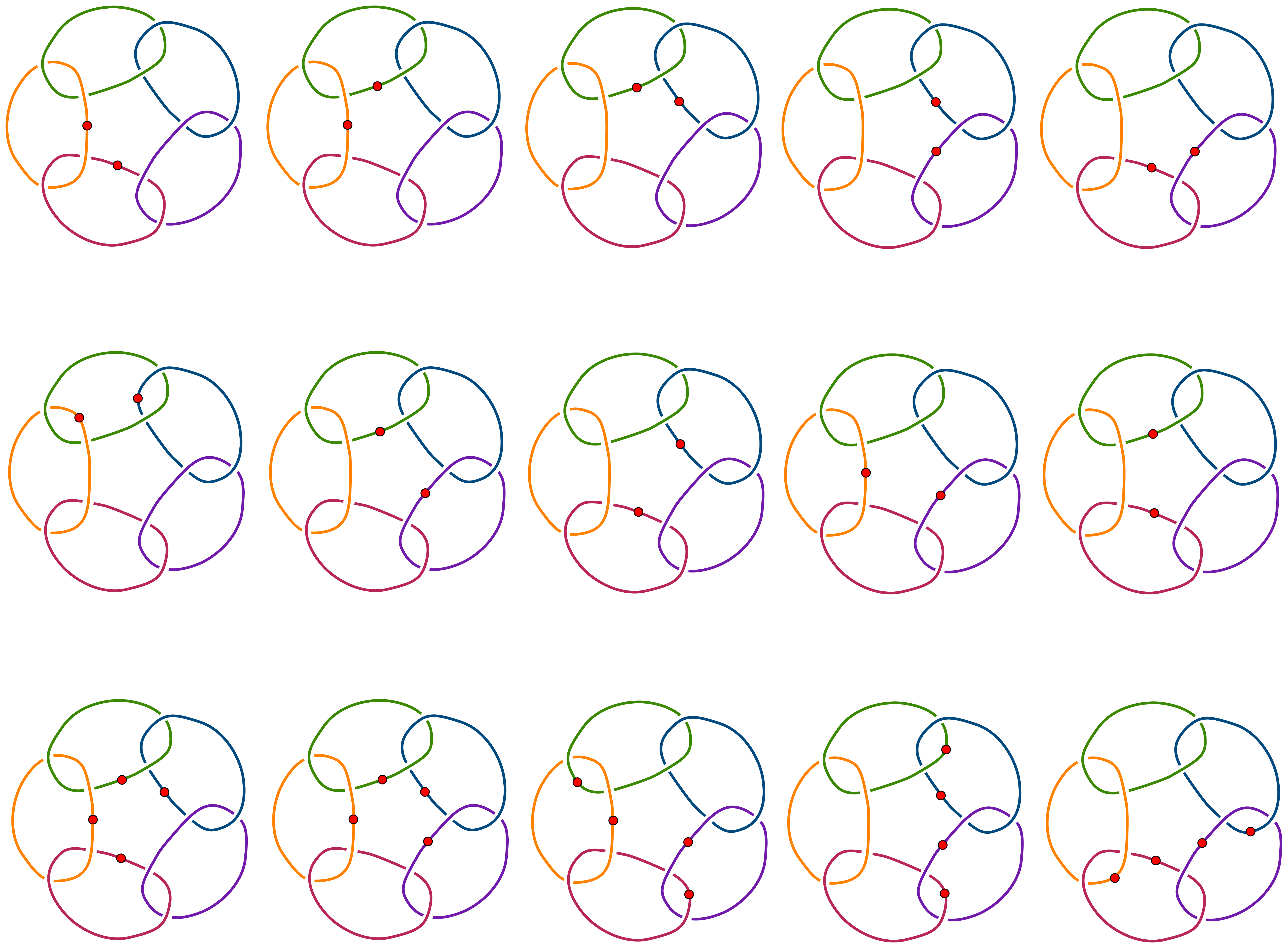}
\caption{Two-fold marked diagrams, one for each non-trivial two-fold marking of L10n113, for which $\text{Hd}(D,\check{\omega})$ is isomorphic to $\kh(L,\omega)$. In each case the $\delta^\#$ betti numbers are $(10,2,10,2)$. For each of the fifteen markings, we first tried the diagram $(D,\check{\omega})$ with dots only on the arcs bounding the interior pentagon region. This is the diagram depicted for the marking if $\text{Hd}(D,\check{\omega})=\kh(L,\omega)$. If that diagram is not depicted, then these groups were not equal. This happened for the first diagram in the second row, and the last three diagrams in the third row.} \label{fig:L10n113}
\end{figure}

In Figure \ref{fig:L10n113} we show dotted diagrams $(D,\check{\omega})$ representing each of the 15 non-trivial two-fold markings for the link L10n113, the last link listed in Table \ref{table}. Each diagram $(D,\check{\omega})$ has the property that $\text{Hd}(D,\check{\omega})$ is isomorphic to $\kh(L,\omega)$. The table indicates that the betti numbers are the same for the 15 cases, and are given by $(10,2,10,2)$.

In the penultimate column we list $\kh(L,\vec{0})_{\delta^\#}$ for the trivial marking data $\omega = \vec{0}$, which is the same as $\overline{\kh}(L;\F)$, the reduced Khovanov homology for $L$ with $\F=\Z_2$-coefficients. In the final column of the table are listed non-thin $\kh(L,\omega)_{\delta^\#}$ for $\omega$ non-trivial. When a 4-tuple of betti numbers $(b_0,b_1,b_2,b_3)$ occurs for $N$ different non-trivial two-fold markings $\omega$, we use the short-hand
\[
    (b_0,b_1,b_2,b_3)\times N.
\]
In the absence of an $N$, we mean that the 4-tuple occurs only once (i.e., $N=1$). Here is an example. The link $L=$ L10n108 has 4 components, and so it has $2^{4-1}=8$ two-fold markings, or equivalently, 8 isomorphism classes of $SO(3)$-bundles over $\Sigma(L)$. For the trivial marking $\omega = \vec{0}$, the $\delta^\# $ (mod 4) graded reduced $\Z_2$ Khovanov homology has betti numbers $(12,0,8,4)$. The table indicates that there are three non-trivial markings $\omega$ for which $\kh(L,\omega)_{\delta^\#}$ is not given by Theorem \ref{thm:tqa}, and these

\begin {table}
\begin{center}
\caption {Non-thin $\text{Kh}(L,\omega)$ for multi-component prime links with up to 10 crossings} \label{table} 
\begin{tabular}{ lccclll } 
\toprule
Link $L$  & $\;\;|L|\;\;$ & $\det(L)$ & $\sigma(L)$ & $\text{Kh}(L,\vec{0})_{\delta^\#}\qquad$ &  Non-thin $\text{Kh}(L,\omega)_{\delta^\#}$ for $\omega\neq \vec{0}$ \\
\midrule
L6n1 & 3 & 4 &  0  &  $(4,0,1,1)$  & none \\ 
L7n1 & 2 & 4 & -5 &   $(3,0,2,1)$ & none \\ 
L8n2 & 2 & 8 & 1 & $(5,1,4,0)$ & none \\ 
L8n3 & 3 & 4 &  -6 &  $(4,0,2,2)$ & $(3,1,3,1)$ \\
L8n6 & 3 & 0 &  -3 &  $(2,1,3,4)$ & $(4,4,4,4)$,  $(2,2,2,2)\times 2$\\
L8n7 & 4 & 16 & 1  &  $(12,0,5,1)$ & none \\
L8n8 & 4 & 0 & 0 & $(1,1,5,5)$  & $(2,2,2,2)\times 7$  \\
L9n1 & 2 & 12 & -5 &  $(7,0,6,1)$ & none\\
L9n3 & 2 & 8 &  -1 &  $(6,2,5,1)$ & none\\
L9n4 & 2 & 4 &  -7 &  $(4,1,3,2)$ & none\\
L9n9 & 2 & 4 &  -3 &  $(5,3,4,2)$ & $(4,2,4,2)$\\
L9n12 & 2 & 4 & 1  &  $(5,2,3,2)$ & $(4,2,4,2)$\\
L9n15 & 2 & 2 & -7 &  $(2,0,2,2)$ & $(2,1,2,1)$\\
L9n18 & 2 & 0 & -6 & $(1,2,3,2)$ & $(2,2,2,2)$\\
L9n19 & 2 & 0 & -4 & $(2,2,3,3)$ & $(2,2,2,2)$\\
L9n21 & 3 & 4 & 0 & $(6,3,3,2)$ & $(4,2,4,2)\times 3$\\
L9n22 & 3 & 20 & 0 & $(12,1,9,0)$ & none\\
L9n23 & 3 & 12 & -4 & $(8,0,6,2)$ & $(7,1,7,1)$\\
L9n25 & 3 & 16 & 0 & $(10,1,7,0)$ & none\\
L9n26 & 3 & 16 & 0 & $(10,1,7,0)$ & none\\
L9n27 & 3 & 0 & -1 & $(3,4,6,5)$ & $(4,4,4,4)\times 3$\\
L10n1 & 2 & 20 & -3 & $(12,1,11,2)$ & none\\
L10n3 & 2& 24 &1  & $(13,1,12,0)$ & none \\
L10n5 & 2 & 16& 3& $(10,2,9,1)$ & none \\
L10n8 & 2 & 16 & 1 & $(9,1,8,0)$ & none \\
L10n9 & 2 & 8 & -1 & $(7,2,6,3)$ & none \\
L10n10& 2 & 12 & -5 & $(9,2,8,3)$ & none\\
L10n13 & 2 &20 & -5& $(11,0,10,1)$ & none\\
L10n14 & 2 & 8 & -1 & $(9,4,7,4)$ & $(8,4,8,4)$\\
L10n18 & 2 & 8 &-1 & $(9,4,8,5)$ & $(8,4,8,4)$\\ 
L10n23 & 2& 16& 5& $(9,1,8,0)$ & none \\
L10n24 & 2 & 16 & -1 & $(9,0,8,1)$ & none \\
L10n25 & 2 & 4 & 1 & $(7,5,6,4)$ & $(6,4,6,4)$\\
L10n28 & 2 & 4 & -3 & $(7,4,5,4)$ &  $(6,4,6,4)$\\
L10n32 & 2 & 0 & 0 & $(4,4,5,5)$ & $(4,4,4,4)$ \\
L10n36 & 2 & 0 & 0 & $(4,4,5,5)$ & $(4,4,4,4)$\\
L10n37 & 2 &12 &1 & $(9,2,7,2)$ & $(8,2,8,2)$\\
\bottomrule
\end{tabular}
\end{center}
\caption*{
have betti numbers $(10,2,10,2)$. Thus the remaining $4$ of the non-trivial markings have $\kh(L,\omega)_{\delta^\#}$ given by Theorem \ref{thm:tqa}, with betti numbers $(8,0,8,0)$.
In the introduction, we speculated that for any two-fold marked link $(L,\omega)$, there exists a diagram $(D,\check{\omega})$ for which $\text{Hd}(D,\check{\omega})$ is isomorphic to $\kh(L,\omega)$. For a certain class of two-fold marking data, our computations suggest a refinement. Let $(D,\check{\omega})$ be a two-fold marked diagram with two dots, where the dots are located near a crossing on adjacent arcs. In other words, $(D,\check{\omega})$ is a dotted surgery diagram with $|\check{\omega}|_1=2$. In this situation, we conjecture that $\text{Hd}(D,\check{\omega})=\kh(L,\omega)$. On the other hand, there are simple dotted diagrams $(D,\check{\omega})$ with two dots for which the dotted diagram homology does not compute the twisted Khovanov homology. This was apparent in Figure \ref{fig:noninvariance}. See also the caption of Figure \ref{fig:L10n113}.
}
\end{table}

\begin{table}
 \label{tab:title} 
\begin{center}
\caption*{Table 1 (Cont.)}
\begin{tabular}{ lccclll  } 
\toprule
Link $L$  & $\;\;|L|\;\;$ & $\det(L)$ & $\sigma(L)$ & $\text{Kh}(L,\vec{0})_{\delta^\#}\qquad$ &  Non-thin $\text{Kh}(L,\omega)_{\delta^\#}$ for $\omega\neq \vec{0}$ \\
 \midrule
L10n39 & 2 & 24 & 5 & $(13,1,12,0)$ & none  \\
L10n42 & 2 & 10 & 5 & $(6,2,6,0)$ &  $(6,1,6,1)$\\
L10n45 & 2 & 6  & -3 & $(4,0,4,2)$ & $(4,1,4,1)$ \\
L10n54 & 2 & 12 & 5 & $(7,1,6,0)$ & none\\
L10n56 & 2 & 0 & -2 & $(3,4,5,4)$ & $(4,4,4,4)$\\
L10n57 & 2 & 0 & 0 & $(4,4,5,5)$ & $(4,4,4,4)$\\
L10n59 & 2 & 0 & 0 & $(4,4,5,5)$ & $(4,4,4,4)$\\
L10n60 & 2 & 12 & 1 & $(9,2,7,2)$ & $(8,2,8,2)$\\
L10n62 & 2 & 24 & 5 & $(13,1,12,0)$ & none\\
L10n66 & 3 & 4 & 0 & $(8,5,5,4)$ & $(6,4,6,4)\times 3$\\
L10n67 & 3 & 36 & 0 & $(20,1,17,0)$ & none\\
L10n68 & 3 & 20 & -6  & $(12,0,10,2)$ & $(11,1,11,1)$ \\ 
L10n70 & 3 & 16 & -2 & $(11,3,9,1)$ & $(10,2,10,2)$ \\ 
L10n72 & 3 & 12 & 2 & $(10,4,8,2)$ & $(9,3,9,3), (8,2,8,2)\times 2$ \\
L10n74 & 3 & 12 & -6 & $(9,1,7,3)$ & $(8,2,8,2)$ \\
L10n77 & 3 & 4 & -8 & $(5,1,3,3)$ & $(4,2,4,2)$\\
L10n82 & 3 & 12 & -2 & $(10,2,6,2)$ & $(8,2,8,2)\times 3$ \\
L10n83 & 3 & 32 & -4 & $(18,0,15,1)$ &  none\\
L10n84 & 3 & 8 & -4 & $(8,4,6,2)$ & $(7,3,7,3), (6,2,6,2)\times 2$ \\
L10n87 & 3 & 8& 2 & $(8,2,4,2)$ & $(6,2,6,2)\times 3$ \\
L10n88 & 3 & 16 & 0 & $(12,3,9,2)$ & $(10,2,10,2)\times 3$\\
L10n91 & 3 & 0 & -3 & $(6,5,7,8)$ & $(6,6,6,6)\times 3$ \\
L10n93 & 3 & 0 & -7 & $(0,2,4,2)$ & $(2,2,2,2)\times 3$ \\
L10n94 & 3 & 0 & 3 & $(2,2,4,4)$ & $(3,3,3,3),(2,2,2,2)\times 2$ \\
L10n97 & 4 & 8 & -1 & $(10,4,4,2)$ & $(6,2,6,2)\times 7$ \\
L10n98 & 4 & 24 & 1 & $(16,2,10,0)$ & $(13,1,13,1),(12,0,12,0)\times 6$\\
L10n101 & 4 & 16 & 3 & $(14,3,7,2)$ & $(10,2,10,2)\times 7$\\
L10n102 & 4  & 16 & -3 & $(13,5,9,1)$ & $(10,2,10,2)\times 7$\\
L10n103 & 4 & 32 & -3 & $(20,0,13,1)$ & none \\
L10n104 & 4 & 0 & -2 & $(2,0,4,6)$ & $(3,3,3,3),(2,2,2,2)\times 6$\\
L10n106 & 4 &  32 & -3 & $(20,0,14,2)$ & $(20,1,14,1)$\\
L10n107 & 4 & 0 & 0 & $(7,7,11,11)$ & $(8,8,8,8)\times 7$\\
L10n108 & 4 & 16 & -5 & $(12,0,8,4)$ & $(10,2,10,2)\times 3$\\
L10n111 & 4 & 0 & 0 & $(6,4,8,10)$ & $(7,7,7,7),(6,6,6,6)\times 6$ \\
L10n112 & 5 & 48 & 2 & $(32,0,17,1)$ & none \\
L10n113 & 5 & 16 & 0 & $(17,1,6,6)$ & $(10,2,10,2)\times 15$\\
 \bottomrule
\end{tabular}
\end{center}
\end{table}

\begin{figure}[t]
\centering
\includegraphics[scale=.30]{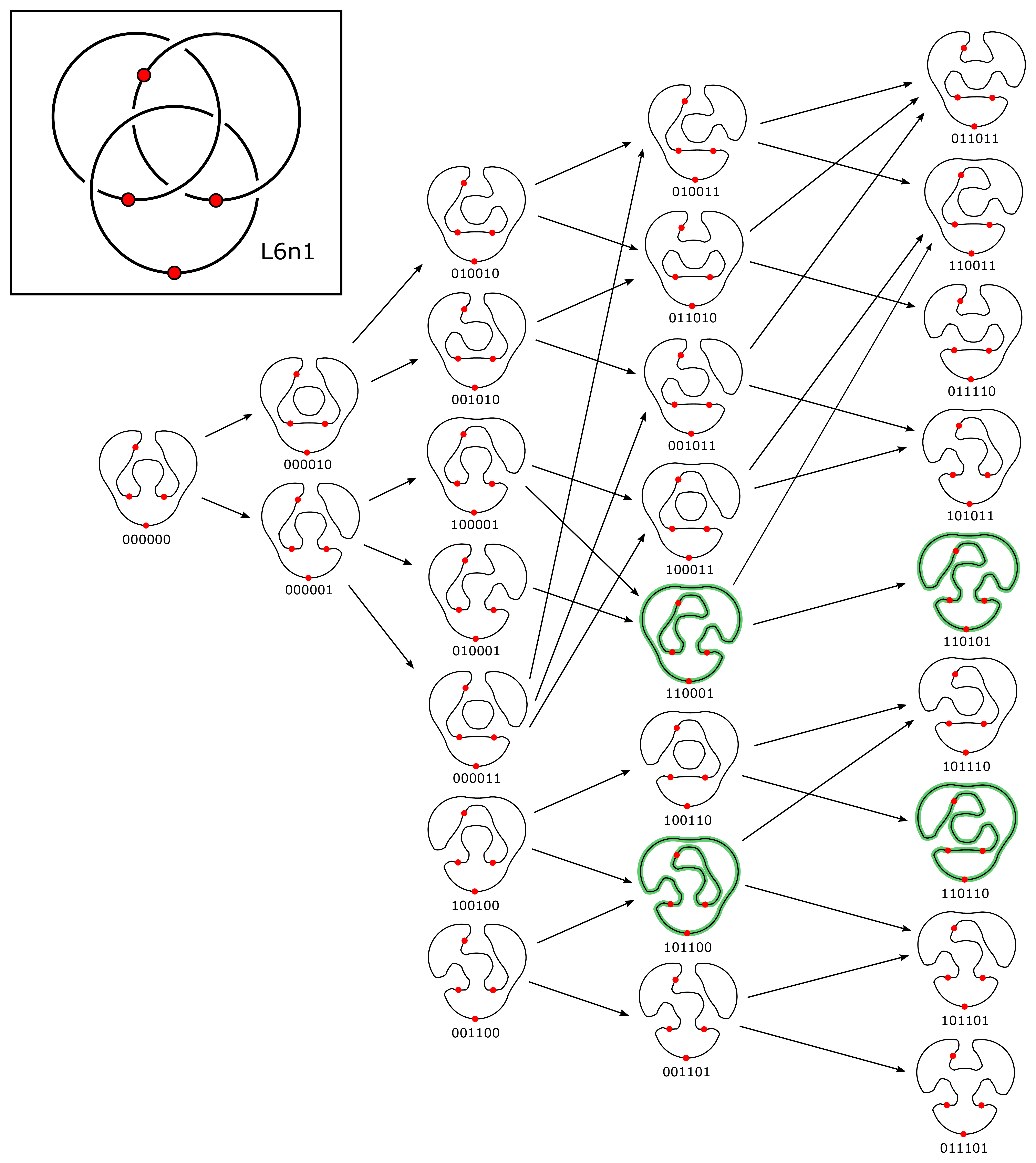}
\caption{The dotted diagram complex that computes $\text{Hd}(D,\check{\omega})$ of the link L6n1 with two-fold marked diagram equivalent to the dotted surgery diagram in Figure \ref{fig:surgerydiagram}. The dotted diagram homology in this case is isomorphic to the twisted Khovanov homology, which is rank 4, supported in $\delta$-grading $0$. Four generators for the homology can be found in the bold (green) resolution diagrams.} \label{fig:L6n1comp}
\end{figure}

\newpage
\bibliography{main.bbl}
\bibliographystyle{alpha}

\vspace{.85cm}

\footnotesize

  \textsc{Department of Mathematics, Brandeis University,
    Waltham, MA}\par\nopagebreak
  \textit{E-mail address:}\;\texttt{cscaduto@brandeis.edu}

\vspace{.35cm}

    \textsc{Department of Mathematics, University of California,
    Los Angeles, CA}\par\nopagebreak
  \textit{E-mail address:}\;\texttt{mstoffregen@math.ucla.edu}

\end{document}